%% file: DalesFein2.tex
\documentclass{amsart}
\usepackage{latexsym}
\usepackage[pdftex]{graphicx}
\usepackage{color}
\usepackage{amssymb}
\usepackage{amsmath}

\def\DD{D^{(1)}}
\def\DDX{\DD(X)}
\def\DDXN{(\DD(X),\norm)}

\newcommand{\wdiff}[1]{{\widetilde{D}^{(1)}(#1)}}

\newcommand{\jpdf}[2]{\includegraphics[width=#1 cm]{#2}}
\newcommand{\Jpdf}[1]{\begin{center}\jpdf{10}{#1}\end{center}}

\newcommand{\Jim}[1]{#1^*}
\def\QX{Q_X}
\def\gammatilde{{\widetilde{\gamma}}}
\def\Cone{C^{(1)}}
\def\Aone{A^{(1)}}
\def\Rplus{\R^+}
\def\Rplusdot{\R^{+\bullet}}
\def\lift{\vskip -\baselineskip}

\numberwithin{equation}{section}

\newtheorem{theorem}{Theorem}[section]

\newtheorem{lemma}[theorem]{Lemma}
\newtheorem{corollary}[theorem]{Corollary}
\newtheorem{proposition}[theorem]{Proposition}

\theoremstyle{definition}
\newtheorem{definition}[theorem]{Definition}
\newtheorem{example}[theorem]{Example}

\theoremstyle{remark}
\newtheorem{remark}[theorem]{Remark}
\newcommand{\gd}{\delta}

\newcommand{\jchange}{}
\newcommand{\jnew}{}

\DeclareMathOperator{\Lip}{Lip}
\DeclareMathOperator{\Log}{Log}
\DeclareMathOperator{\Arg}{Arg}
\DeclareMathOperator{\Tint}{int}

\def\smallskip{\vskip 0.25cm}
\def\medskip{\smallskip}

\def\phi{\varphi}
\newcommand{\dd}{\,{\rm d}}
\def\C{{{\mathbb C}}}
\def\R{{{\mathbb R}}}
\def\T{{{\mathbb T}}}

\def\N{{{\mathbb N}}}

\def\I{{{\mathbb I}}}
\def\sskip{\vskip 0.3 cm}
\def\jskip{}

\def\i{{\rm i}}
\def\e{{\rm e}}

\def\Jnorm{{\|\,{\cdot}\,\|}}
\def\jnorm#1{{\|{#1}\|}}
\def\norm{\Jnorm}
\def\unorm#1{{|{#1}|}}
\def\Unorm{{|\,{\cdot}\,|}}
\def\eps{\varepsilon}
\newcommand{\QED}{\hfill$\Box$\par}
\newcommand {\D}{\mathbb D}

\def\stddiff#1{\DD(#1)}

\def\BF#1{D_{\paths}^{(1)}(#1)}
\def\Mone{M_{z_0}^{(1)}(X)}

\newcommand{\paths}{\mathcal{F}}

\newcommand{\Length}[1]{|#1|}
\newcommand{\lv}{\left\vert}
\newcommand{\rv}{\right\vert}
\newcommand{\lV}{\left\Vert}
\newcommand{\rV}{\right\Vert}

\def\rd{{\rm d}}

\def\sI{\mathcal{I}}

\begin{document}

\title[Normed algebras of differentiable functions]{Normed algebras of differentiable functions on compact plane sets}
\author{H.\ G.\ Dales}
\address{Department of Pure Mathematics, University of Leeds, Leeds~LS2~9JT, UK}
\email{garth@maths.leeds.ac.uk}
\author{J.\ F.\ Feinstein}
\address{School of Mathematical Sciences, University of Nottingham, University Park, Notting\-ham~NG7~2RD, UK}
\email{Joel.Feinstein@nottingham.ac.uk}

\subjclass[2000]{Primary 46H05, 46J10; Secondary 46E25}

\date{}

\begin{abstract}
We investigate the completeness and completions of the normed algebras $\DDXN$  for perfect, compact plane sets $X$.
In particular, we construct
a radially self-absorbing, compact plane set $X$ such that the normed algebra $\DDXN$  is not complete. This solves a question of Bland and Feinstein.
We also prove that there are several classes of connected, compact plane sets $X$ for which the completeness of $\DDXN$  is equivalent to the pointwise regularity of $X$. For example, this is true for all rectifiably connected, polynomially convex, compact plane sets with empty interior, for all star-shaped, compact plane sets, and for all Jordan arcs in $\C$.

In an earlier paper of Bland and Feinstein, the notion of an $\paths$-derivative of a
function was introduced, where $\paths$ is a suitable set of
rectifiable paths, and with it a new family of Banach algebras  $\BF{X}$
corresponding to the normed algebras $\DD(X)$.
In the present paper, we obtain stronger results
concerning the questions when $\DD(X)$ and $\BF{X}$ are equal, and when the former
is dense in the latter. In particular, we show that equality holds
whenever $X$ is \lq $\paths$-regular'.

An example of Bishop shows that the completion of $\DDXN$ need not be
semisimple. We show that the completion of $\DDXN$ is semisimple
whenever the union of all the rectifiable Jordan arcs in $X$
is dense in $X$.

We prove that the character space of $\DD(X)$ is equal to $X$ for all perfect, compact plane sets $X$, whether or not
$\DDXN$   is complete. In particular, characters on the normed algebras $\DDXN$ are automatically continuous.
\end{abstract}

\maketitle

\section{Introduction}
\noindent
Throughout this paper, by a \textit{compact space} we shall mean a non-empty, compact, Hausdorff topological space; by a \textit{compact plane set} we shall mean a \emph{non-empty}, compact subset of the complex plane. Recall that such a set is \emph{perfect} if it has no isolated points.

Let $X$ be a perfect, compact plane set, and let $\DD(X)$ be the normed algebra of all continuously differentiable, complex-valued functions on $X$. These algebras (and others) were discussed by Dales and Davie in \cite{Dales-Davie}.
We shall continue the study of these algebras, concentrating on the problem of giving necessary and sufficient conditions on $X$ for $\DDXN$ to be complete.

We shall see an interesting relationship between the geometry of a compact plane set $X$ and the properties of the normed algebra $\DDXN$. For example, it is shown in
Corollary \ref{poly-term-eliminated}
that, for polynomially convex, geodesically bounded (and hence connected), compact plane sets $X$, the normed algebra $\DDXN$ is complete if and only if the following condition (condition (\ref{poly-condition})) holds:
for each $z \in X$, there exists $B_z > 0$ such that, for all polynomials $p$ and all $w \in X$, we have
\[
|p(z)-p(w)| \leq B_z |p'|_X|z-w|\,.
\]
However, for many such sets $X$, it is far from easy to determine whether or not this condition, which involves only polynomials, is satisfied.
We conjecture that, for each connected, compact plane set $X$ with more than one point, $\DDXN$ is complete if and only if $X$ is pointwise regular.
We shall prove that this is true for a variety of somewhat exotic connected, compact plane sets. We shall also describe some other connected, compact plane sets for which we have not been able to determine whether or not this is the case.

In \cite{Bland-Fein}, the notion of $\paths$-derivative was introduced in order
to describe the completions of the  spaces $\DDXN$;
here $\paths$ is a suitable set of rectifiable paths lying in a compact plane set $X$.
We shall remove an unnecessary restriction on the sets of paths considered in \cite{Bland-Fein}.
We shall also introduce a generalization of
the standard notion of pointwise regularity for compact plane sets in order to further strengthen
some of the results of \cite{Bland-Fein}.

We shall then investigate further the completeness of the spaces $\DDXN$ for compact plane sets $X$,
and discuss their completions.
In particular, we shall answer a problem raised in \cite{Bland-Fein} by
 constructing a radially self-absorbing, compact plane set $X$ such that $\DDXN$ is not complete.

\section{Preliminary concepts and results}
\noindent
We begin with some standard terminology, notation, definitions and results.
For more details, the reader may wish to consult \cite{Dales}.

We denote the unit interval $[0,1]$ by $\I$ and the complex plane by $\C$; the real line is $\R$,
$\Rplus = [0,\infty)$, and
$\Rplusdot=(0,\infty)$; we often identify $\C$ with $\R^2$.

We denote by $\Cone(\R^2)$ \jchange the algebra of functions from $\R^2$ to $\C$ which have continuous first-order partial derivatives on $\R^2$.
We denote by $Z$ the coordinate functional on $\C$,
$z \mapsto z$, or the restriction of this function to some subset
of $\C$.  Similarly we write $1$ for
the function constantly equal to $1$.

We recall the following standard notation. Let $(a_n)$ be a sequence in $\Rplus$, and let $(b_n)$ be a sequence
in $\Rplusdot$.
We write
\[
a_n=O(b_n)
\]
if the sequence $(a_n/b_n)$ is bounded, and
\[
a_n=o(b_n)\quad\text{as}\:\,n\to\infty
\]
if \jchange $a_n/b_n \to 0$ as $n \to \infty$.

Let $X$ be a compact space.
We denote the algebra of all continuous, complex-valued functions on
$X$, with the pointwise algebraic operations, by $C(X)$. For $f \in C(X)$, we denote the uniform norm of $f$ on a non-empty
subset $E$ of $X$ by $|f|_E$.  Thus $(C(X), \lv\,\cdot\,\rv_X)$ is  a commutative Banach algebra; see \cite[\S4.2]{Dales}.

\begin{definition}
Let $X$ be a compact space.
A \textit{normed function algebra} on $X$ is a normed algebra $(A,\Jnorm)$ such that $A$ is a subalgebra
of $C(X)$, such that $A$ contains the constants and separates the points of $X$, and such that, for all $f \in A$, we have
$\jnorm{f} \geq \unorm{f}_X$.
A \textit{Banach function algebra} on $X$ is a normed function algebra $(A,\Jnorm)$ on $X$ such that
 $(A,\Jnorm)$ is complete.
A \textit{uniform algebra} on $X$ is a Banach function algebra $A$ on $X$ such that the norm of $A$ is equivalent to the uniform norm $\Unorm_X$.\smallskip

Of course, in the case where $(A,\Jnorm)$ is a Banach algebra and a subalgebra of $C(X)$, it is automatic that $\jnorm{f} \geq \unorm{f}_X$ for all $f \in A$.

Let $A$ be a complex algebra. As in \cite{Dales}, the space of all characters on $A$ is denoted by $\Phi_A$. If $A$ is a normed algebra, the space of continuous characters on $A$ is $\Psi_A$; $\Psi_A$ is a locally compact space with respect to the relative weak-$*$ topology.
In the case where $A$ is a normed function algebra on  a compact space $X$, we
define
\[
\eps_x : f \mapsto f(x)\,,\quad A\rightarrow \C\,,
\]
for each $x \in X$. Then $\eps_x \in \Psi_A$, and the map $x \mapsto \eps_x$,
$ X \rightarrow \Psi_A$, is a continuous embedding. As in \cite[\S4.1]{Dales}, we say that $A$ is \textit{natural} (on $X$) if this map is surjective.
\end{definition}

The following result is due to Honary \cite[Theorem]{Honary}.\smallskip

\begin{proposition} \label{Honary} Let $X$ be a compact space, and let $(A, \norm)$ be a normed function algebra on $X$, with uniform closure $B$. Then $A$ is natural on $X$ if and only if both of the following conditions hold:
\smallskip
{\rm (a)} $B$ is natural on $X$;\smallskip

{\rm (b)}  $\lim_{n\to\infty} \lV f^n\rV^{1/n} = 1$ for each $f \in A$ with $\lv f \rv_X =1$. \qed
\end{proposition}
\smallskip
We now discuss (complex) differentiability for functions defined on compact plane sets.

\begin{definition}
Let $X$ be a perfect, compact plane set. A function $f \in C(X)$ is \textit{differentiable
at} a point $a\in X$ if
the limit
\[ f'(a) = \lim_{z\to a, \ z\in X}
\frac{f(z)-f(a)}{z-a} \]
exists.
\end{definition}

We call $f^\prime(a)$ the \textit{{\rm (}complex{\rm)} derivative} of $f$ at $a$.
Using this concept of derivative, we define the terms
\textit{differentiable on} $X$ and
\textit{continuously differentiable on} $X$
in the obvious way, and we denote the set of continuously differentiable
functions on $X$ by $\DD(X)$.
For $f \in \DD(X)$, set
\[
\|f\|=|f|_X + |f'|_X\,.
\]
Then $(\DD(X),\Jnorm)$ is immediately seen to be a normed function algebra on $X$.
We denote the completion of $\DDXN$ by $\wdiff{X}$.

The normed function algebra $\DDXN$ is often incomplete, even for fairly nice sets $X$.
For example, in \cite[Theorem 3.5]{Bland-Fein}, Bland and Feinstein gave an example of a rectifiable Jordan arc $J$
(as defined below) such that $(\DD(J),\norm)$ is incomplete. In the same paper it was shown
\cite[Theorem 2.3]{Bland-Fein} that
$\DDXN$ is incomplete whenever $X$ has infinitely many components; this last result was also proved in \cite[(3.1.10)(iii)]{Dales-Thesis}.
We shall give several further examples and results later when we investigate necessary and sufficient conditions for the completeness of
these normed algebras.

We now recall the standard definitions of pointwise regularity and
uniform regularity for compact plane sets.
We shall suppose that the reader is familiar with the elementary results and
definitions concerning rectifiable paths, including integration of continuous,
complex-valued functions along such paths; for more details see, for
example, Chapter 6 of \cite{Apostol}.

\begin{definition}
A \textit{path} in $\C$ is a
continuous function  $\gamma : [a,b] \rightarrow \C$,
where $a$ and $b$ are real numbers with $a<b\,$;
$\gamma$ is a path \textit{from
$\gamma(a)$ to $\gamma(b)$} with \textit{endpoints}
$\gamma^-=\gamma(a)$ and $\gamma^ +=\gamma(b)$; in this case, $\gamma^-$ and $\gamma^+$ are  \textit{connected by} $\gamma$.
We denote by $\Jim{\gamma}$ the image $\gamma([a,b])$ of $\gamma$.
A \textit{subpath} of $\gamma$ is any path obtained by restricting
$\gamma$ to a non-degenerate, closed subinterval of $[a,b]$.  A path in $\C$ is \textit{admissible} if it is rectifiable and has no
constant subpaths.  For a subset  $X$ of $\C$, a \textit{path in $X$} is a path $\gamma$ in $\C$ such that
$\Jim{\gamma} \subseteq X$. We also say that such a path is \emph{a path in $X$ from $\gamma^-$ to $\gamma^+$}, and that $\gamma^-$ and $\gamma^+$ are \emph{connected in $X$ by $\gamma$}.
\smallskip
The length of a rectifiable path $\gamma$ will be denoted by $|\gamma|$.
The length of a non-rectifiable path is defined to be $\infty$.
\end{definition}

Note that we
distinguish between a path $\gamma$ and its image $\Jim{\gamma}$. This is because it is possible for two very different paths to have the same image. There is a lack of consistency in the literature over the usage of the terms `path', `curve', and `arc'. For us, an \emph{arc in $X$} \jchange is the image of a \emph{non-constant} path in $X$.
A \emph{rectifiable arc in $X$} is the image of a non-constant, rectifiable path in $X$. A \emph{Jordan path in} $X$ is a path $\gamma$ in $X$ such that $\gamma$ is injective; a \emph{Jordan arc in} $X$ is the image of a Jordan path in $X$.
We define \emph{admissible arc in $X$} \jchange and \emph{rectifiable Jordan arc in $X$} similarly.

Let $X$ be a compact plane set, and let $z,w \in X$ with \jchange $z \neq w$. Suppose that $z$ and $w$ are connected in $X$ by a path $\gamma$. Then there is also a Jordan path in $X$ from $z$ to $w$.
This does not appear to be immediately obvious; see \cite[Problem 6.3.12(a)]{Engelking}.

Let $\gamma$ be a non-constant, rectifiable path in $\C$, with length $L$. Although $\gamma$ need not itself be admissible, nevertheless there is always a path $\gammatilde:[0,L]\rightarrow \C$ such that $\gammatilde$ is admissible, $\gammatilde$ has the same endpoints and image as $\gamma$, and $\gammatilde$ is parametrized by arc length. (See, for example,
\cite[pp.~109--110]{Federer}.)
Such a path $\gammatilde$ is necessarily Lipschitzian, with Lipschitz constant $1$, in the sense that
$|\gammatilde(s)-\gammatilde(t)| \leq |s-t|$ whenever $0 \leq s < t \leq L$.
In particular, $\Jim{\gammatilde}$, and hence also $\Jim{\gamma}$, must have zero area.

Recall that a path $\gamma=\alpha+\i\beta$ (where $\alpha$ and $\beta$ are real-valued) is rectifiable if and only if both $\alpha$ and $\beta$ are of bounded variation \cite[Theorem 6.17]{Apostol}.
In this case \[\int_\gamma f = \int_\gamma f(z)\,\rd z\] is defined as a Riemann--Stieltjes integral for all $f \in C(\Jim{\gamma})$ \cite[Theorems 7.27 and 7.50 (see also p. 436)]{Apostol}.

\begin{definition}
Let $X$ be a compact plane set, and let $z,w \in X$ be points which are connected by a rectifiable path in $X$. Then
\[
 \delta (z,w) = \inf\{\lv \gamma \rv: \gamma\text{ is a rectifiable path from $z$ to $w$ in $X$}\}\,.
\]
We call $\delta(z,w)$ the \jnew \emph{geodesic distance} between $z$ and $w$ in $X$.
The set  $X$ is \textit{rectifiably connected} if,
 for all $z$ and $w$ in $X$, there is a rectifiable path $\gamma$ connecting  $z$ to $w$ in $X$.
\end{definition}\smallskip

Note that every rectifiably connected, compact plane set with more than one point is perfect.

Suppose that $X$ is rectifiably connected.
Then $\delta$ is certainly a metric on $X$, \jnew and it is well-known that the infimum in the definition of $\delta$ is always attained. Thus, for each pair of distinct points $z,w \in X$, there is a rectifiable path $\gamma$ connecting $z$ to $w$ in $X$ such that $|\gamma|=\delta(z,w)$. In this case it is clear that $\gamma$ is a Jordan path in $X$.

A rectifiably connected, compact plane set $X$ is \textit{geodesically bounded} if $X$ is
bounded with respect to the metric $\delta$.
In this case the \textit{geodesic diameter} of $X$ is defined to be
\[
\sup\{\delta(z,w): z,w\in X\}\,.
\]
Easy examples show that rectifiably connected, \jchange compact plane sets need not be geodesically bounded.

\begin{definition}
Let $X$ be a compact plane set. For $z \in X$, the set $X$ is \textit{regular at} $z$
if there is a constant $k_{z} > 0$ such that,
for every $w \in X$, there is a rectifiable path $\gamma$ from $z$ to $w$ in $X$ with
$\Length{\gamma} \leq k_{z}|z-w|$.

The set $X$ is \textit{pointwise regular}
if $X$ is regular at every point $z \in X$, and $X$ is \textit{uniformly regular} if,
further, there is one constant $k > 0$ such that, for all $z$ and $w$ in
$X$,
there is a rectifiable path $\gamma$ from $z$ to $w$ in $X$ with
$\Length{\gamma} \leq k|z-w|$.
\end{definition}
\smallskip
Note that every convex, compact plane set is obviously uniformly regular and every pointwise regular, compact plane set is geodesically bounded.

Dales and Davie \cite[Theorem 1.6]{Dales-Davie} showed that $\DDXN$ is complete whenever
$X$ is a finite union of
uniformly regular, compact plane sets.
However, as observed in \cite{Dales-Thesis} and \cite{Honary-Mahyar1}, the proof given
in \cite{Dales-Davie}
is equally valid for pointwise regular, compact plane sets.
This gives the following result.

\begin{proposition}
\label{pr-sufficient}
Let $X$ be a finite union of pointwise regular, compact plane sets.
Then $(\DD(X),\Jnorm)$ is complete.
\QED
\end{proposition}
\smallskip
One purpose (unfortunately not achieved) of the present paper is to decide whether or not the converse of
Proposition \ref{pr-sufficient} holds true.

Note that whenever a \emph{connected} compact, plane set $X$ is a finite union of pointwise regular, compact plane sets, $X$ itself is already pointwise regular. The corresponding statement concerning uniform regularity is, however, false.

\begin{definition}
Let $(X,d)$ be a compact metric space with more than one  point.  Then a function $f\in C(X)$ belongs to the \emph{Lipschitz space}, ${\rm Lip\,}X$, if
\[
p(f) : = \sup\left\{\frac{\lv f(z)-f(w)\rv}{d(z,w)} : z,w \in X, \,z\neq w\right\} < \infty\,.
\]
\end{definition}
It is standard that ${\rm Lip\,}X$ is, in fact, a natural Banach function algebra on $X$ with respect to the norm $\norm$ given by
\[
\lV f \rV = \lv f \rv_X + p(f)\quad (f \in {\rm Lip\,}X)\,.
\]
For details of these algebras, and of their relatives ${\rm lip\,}X$, see for example \cite[\S 4.4]{Dales} and \cite{Weaver}.

Let $X$ be a compact plane set $X$. We give $X$ the usual
Euclidean metric
\[
d(z,w)=|z-w|\quad(z,w \in X)\,,
\]
and define the Banach function algebra $\Lip X$ accordingly.
Now suppose, in addition, that $X$ is rectifiably connected. Then it is noted in
\cite[Lemma 1.5(i)]{Dales-Davie} that
\[
\lv f(z)-f(w)\rv  \leq \delta(z,w)\lv f'\rv_X\quad (z,w\in X,f \in D^{(1)}(X))\,.
\]
Thus, in the case where $X$ is uniformly regular, $D^{(1)}(X)$ is a closed sub\-algebra of $({\rm Lip\,}X, \norm)$.

\smallskip
We now recall the definitions of the uniform algebras $P(X)$, $R(X)$, and $A(X)$ for compact plane sets $X$, and
some standard results concerning these algebras. We refer the reader to \cite[\S4.3]{Dales}, \cite[Chapter II]{Gamelin} and \cite{Stout} for further details.

Let $X$ be a compact plane set.
The \emph{polynomially convex hull} of $X$, denoted by $\widehat X$, is the complement of the unbounded component of
$\C \setminus X$.
The \emph{outer boundary} of $X$ is the boundary of $\widehat X$.
The set $X$ is \emph{polynomially convex} if $\C \setminus X$ is connected.

The spaces of restrictions to $X$ of the polynomial functions and of the rational functions with poles off $X$ are denoted by $P_0(X)$ and $R_0(X)$, respectively.
The closures of these spaces in $(C(X), \lv \,\cdot\,\rv_X) $ are the uniform algebras $P(X)$ and $R(X)$, respectively. The algebra $R(X)$ is always natural; the character space of $P(X)$ is identified with $\widehat{X}$, and so $P(X)$  is natural if and only if
$X= \widehat{X}$, i.e., if and only if  $X$ is polynomially convex.

For a non-empty, open subset $U$ of $\C$, we write $O(U)$ for the algebra of analytic functions on $U$.

Let  $X$ be a compact plane set. We denote by $O(X)$ the set of restrictions to $X$ of functions which are analytic on some neighbourhood of $X$. Thus $g \in O(X)$ if and only if there are an open neighbourhood $U$ of $X$ and a function $f \in O(U)$ with $f|_X = g$.

Now let $X$ be a compact plane set with interior $U$. Then $A(X)$ is the uniform algebra of all continuous functions on $X$ such that $f\mid U \in O(U)$.  By a theorem of Arens (see  \cite[Theorem 4.3.14]{Dales} or \cite[Chapter II, Theorem 1.9]{Gamelin}), $A(X)$ is a natural  uniform algebra on $X$.  It may be that $R(X)\subsetneq A(X)$ (see \cite[Chapter VIII, \S8]{Gamelin}); however  $R(X) = A(X)$ whenever $\C \setminus X$ has only finitely many components. Let $A$ be a uniform algebra on $X$ with $R(X) \subseteq A \subseteq A(X)$.  Then we do not know whether or not $A$ is necessarily natural; if this were always the case, then some of our later open questions would be easily resolved.

In the case where $X$ is polynomially convex, Mergelyan's theorem \cite[Chapter II, Theorem 9.1]{Gamelin} tells us that $P(X)=A(X)$. In particular, when $X$ is polynomially convex and has empty interior, we have $P(X)=C(X)$. (This latter fact is Lavrentiev's theorem \cite[Chapter II, Theorem 8.7]{Gamelin}.)
\sskip
We conclude this section by recalling the definition
of some related spaces which were discussed
in \cite{Bland-Fein}.

Note that it is obvious that whenever $X$ is a compact plane set such that $\Tint X$ is dense in $X$,
then $X$ is perfect.\smallskip

\begin{definition}
Let $X$ be a compact plane set such that $\Tint X$ is dense in $X$.
Set $U=\Tint X$.
Then $\Aone(X)$ is the set of functions $f$ in $A(X)$ such that
$(f|_U)'$ extends continuously to the whole of $X$. In this case we set
\[
\|f\|=|f|_X + |f'|_U\,\quad(f \in \Aone(X))\,.
\]
\end{definition}\smallskip

Clearly $\DDX \subseteq \Aone(X)$, and $(\Aone(X),\Jnorm)$ is a normed function algebra on $X$.
Moreover, it is easy to see that $(\Aone(X),\Jnorm)$ is complete, and hence a Banach function algebra on $X$.

In the following sections, we shall discuss the relationships
 between all of the algebras so far discussed.

\section{Inclusion relationships between the algebras}
\noindent
Let $X$ be a perfect, compact plane set.
Certainly we have the inclusions
\[
P_0(X) \subseteq R_0(X) \subseteq O(X) \subseteq \DD(X) \subseteq A(X)\,.
\]
In particular,
$(Z-w1)^{-1}$ belongs to $\DDX$ whenever $w \in \C \setminus X$.

It is an elementary consequence of Runge's theorem \cite[p.~35]{Bonsall-Duncan} that the closures of $R_0(X)$ and $O(X)$ in $(\DD(X),\norm)$ are always the same. Similarly, if $X$ is polynomially convex, then both of these $\norm$-closures are equal to the closure of $P_0(X)$ in $(\DD(X),\norm)$. (Part of this was noted on page 106 of \cite{Bland-Fein}.)

We do not know whether
or not we always have $D^{(1)}(X)\subseteq  R(X)$.
Nor do we know whether or not $R_0(X)$ is always dense in $(\DD(X),\Jnorm)$.
Since we have $D^{(1)}(X)\subseteq  A(X)$, clearly
$D^{(1)}(X)\subseteq  R(X)$ whenever $R(X) = A(X)$.

We do have the following easy result from \cite[Lemma 1.5]{Dales-Davie}.\smallskip

\begin{proposition}
\label{sub-R(X)}
Let $X$ be a uniformly regular,
compact plane set.
Then the inclusion $D^{(1)}(X)\subseteq  R(X)$ holds.
\QED
\end{proposition}

Theorem 3 of \cite{Honary-Mahyar2} appears to claim that $D^{(1)}(X)\subseteq  R(X)$ for each perfect, compact plane set, or perhaps for each
 perfect, compact plane set $X$ such that $X$ is pointwise  regular.
 However the proof is based on an invalid use of Whitney's extension theorem \cite[Chapter I, Theorem 3.2]{Malgrange}  in an attempt to show that every function in $\DDX$ has an extension which lies in $\Cone(\R^2)$. \jchange (Here we are identifying $\C$ with $\R^2$ in the usual way.)
The following example, which is a modification of \cite[Example 2.4]{Bland-Fein}, shows that functions in $\DDX$ need not have such extensions, even in the case where $X$ is a pointwise regular Jordan arc.

\begin{figure}[htb]
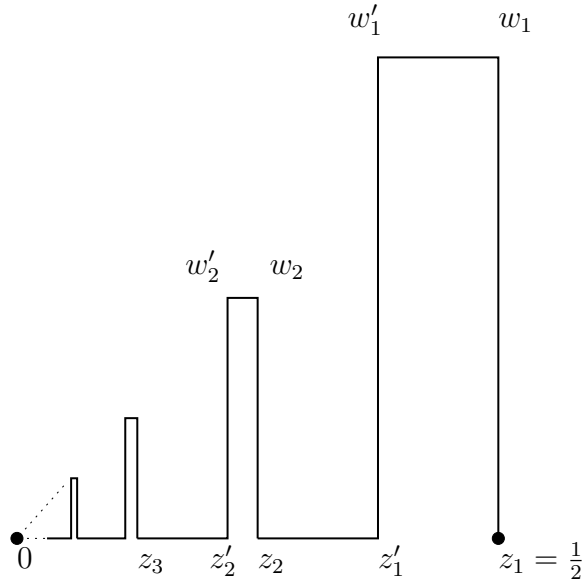
\caption{The arc $J$ constructed in Theorem \ref{bad-arc} (not to scale)}
\label{counter-mh}
\end{figure}

 \begin{theorem}
 \label{bad-arc}
There exist a pointwise  regular Jordan arc $J$ and a function $f$ in $D^{(1)}(J)$ such that there is no function $F$ in \jchange $\Cone(\R^2)$ with $F|_J = f$.
\end{theorem}
\begin{proof}
An example of such an arc $J$ is shown in Figure \ref{counter-mh}, above.

For $n \in \N$, set
\[
z_n = x_n = 2^{-n}\,,\quad z_n'=2^{-n}-2^{-3n}\,,
\]
and \jchange
\[
w_n= 2^{-n} + 2^{-n}\,\i\,,\quad w_n'= 2^{-n}-2^{-3n} + 2^{-n}\,\i\,.
\]
Note that we have
\[
z_n'=z_n-2^{-3n}=x_n-2^{-3n}\,,\quad w_n=z_n+2^{-n}\,\i\,,
\]
and
\[
w_n'=z_n'+2^{-n}\,\i = w_n - 2^{-3n}\,.
\]
Let $J_n$ be the Jordan arc made up of the four straight lines joining, successively, $z_n$, $w_n$, $w_n'$, $z_n'$, and $z_{n+1}$. It is clear that we may obtain a pointwise regular Jordan arc $J$ by gluing together all of the arcs $J_n$ $(n \in \N)$, and then adding in the point $0$.

For $n \in \N$, set $c_n=x_n/n$. We now \emph{claim} that there exists $f \in \DD(X)$ such that, for all $n \in \N$, $f(z_n)= c_n$ while
$f(z_n') = c_{n+1}$. This may be achieved by setting $f(0)=0$, and defining $f$ on each $J_n$ separately as follows.
For $n \in \N$, let $V_n$ be the straight line joining $z_n$ to $w_n$. For $z \in J_n \setminus V_n$, set
$f(z)=c_{n+1}$. For $z=x_n+\i y \in V_n$, set
$f(z)=a + b \cos(2^n \pi y)$, where $a= (c_n+c_{n+1})/2$  and $b = (c_n - c_{n+1})/2$.
It is now easy to check that $f$ satisfies the conditions of the claim.

Obviously, the restriction to $J$ of any function \jchange in $\Cone(\R^2)$ must be in $\Lip J$. However, we have
\[
\frac{|f(z_n)-f(z'_n)|}{|z_n-z'_n|} = \frac{c_n-c_{n+1}}{2^{-3n}} = 2^{2n-1} \frac{n+2}{n(n+1)} \to \infty\quad\text{as}\:\,n \to \infty\,,
\]
and so these quotients are unbounded. Thus $f$ is not in $\Lip J$, and it follows that $f$ has no extension in \jchange $\Cone(\R^2)$, as required.
\end{proof}
Of course, for the arc $J$ constructed in Theorem \ref{bad-arc}, we do have
\[
\DD(J) \subseteq P(J)=R(J)=C(J)\,.
\]
\sskip
Let $X$ be a compact plane set with dense interior. It is obvious that the inclusion $\DD(X) \subseteq \Aone(X)$ is an isometric inclusion of $\DDXN$ in $(\Aone(X),\norm)$.
Thus the completion of $(\DD(X),\Jnorm)$, namely $\wdiff{X}$, is just the $\norm$-closure of $\DDX$ in $\Aone(X)$.

We do not know whether or not we always have $\Aone(X)\subseteq R(X)$. However, it is obvious that $\Aone(X) \subseteq A(X)$, and so we have $\Aone(X) \subseteq R(X)$ whenever $R(X)=A(X)$.
We shall see later, in Theorem \ref{Cantor-squares}, that $\DDX$, and hence also $R_0(X)$, need not be dense in $(\Aone(X),\Jnorm)$.

\section{Naturality}
\noindent We next show that $\DD(X)$ is natural for every perfect, compact plane set $X$. In addition, every character on $(\DD(X),\Jnorm)$ is continuous. The elementary proof is based on the method used by Jarosz in \cite{Jarosz} to prove that a certain Banach function algebra ${\rm Lip}_{Hol}(X,\alpha)$ is natural.

\smallskip
\begin{theorem}
Let $X$ be a perfect, compact plane set, and let $A$ be the normed function algebra $(\DD(X),\Jnorm)$.
Then $A$ is natural on $X$, and $\Phi_A=\Psi_A$.
\end{theorem}

\begin{proof}
Let $\phi \in \Phi_A$, and set $w=\phi(Z)$.  Then $\phi(Z-w1) = 0$, and so $Z-w1$ is not invertible in $A$. Since $R_0(X) \subseteq A$, it follows that
$w \in X$.

We shall show that $\phi = \eps_w$. To see this, it is sufficient to prove the inclusion $\ker \eps_w  \subseteq \ker \phi$.
Take $f \in A$ with $f(w)=0$. Since $f$ is differentiable at $w$,
 there is a positive constant $C$ such that, for all $z \in X$, we have
$|f(z)|\leq C |z-w|$.
It is now easy to see that
$f^3=(Z-w)g$ for a (unique) function $g \in \DD(X)$ (with $g(w)=g'(w)=0$).
This gives \[(\phi(f))^3 = \phi(f^3) = \phi(Z-w)\phi(g)=0\,,
\]
and so $\phi(f)=0$.
The result follows.
\end{proof}

\smallskip

Let $X$ be a perfect, compact plane set with dense interior. We do not know \jchange whether or not $\Aone(X)$ is always natural.
Let $f$ be a function in $\Aone(X)$ with $\lv f \rv_X =1$.  We see that $1 \leq \lV f^n\rV \leq 1 + n|f'|_X =O(n)$, and so $\lim_{n\to\infty} \lV f^n\rV^{1/n} = 1$. Thus,
by Proposition \ref{Honary}, $\Aone(X)$ is natural if and only if its uniform  closure is natural. Clearly \jchange this is the case whenever $R(X)=A(X)$.

\section{$\paths$-derivatives}
\noindent
We noted earlier that, whenever $X$ is a compact plane set with dense interior, $\wdiff{X}$ is the $\norm$-closure of $\DDX$ in $\Aone(X)$. The condition that $\Tint X$ be dense in $X$ is too restrictive for our purposes, and so we now introduce a new class of compact plane sets.

\begin{definition}
Let $X$ be a compact plane set.
Then $X$ is \textit{semi-rectifiable} if the union of all the rectifiable Jordan arcs
in $X$ is dense in $X$.
\end{definition}
Clearly every semi-rectifiable, compact plane set is perfect.

Although an admissible arc need not be a Jordan arc, it is never\-theless easy to see that a compact plane set is semi-rectifiable if and only if the union of all the admissible arcs
in $X$ is dense in $X$. In view of our earlier comments, this condition is also equivalent to the the condition that the union of all the rectifiable arcs in $X$ be dense in $X$.

We now define another new term, \lq effective'. This is a modification of the term \lq useful', which was introduced in \cite{Bland-Fein}.
In particular, we shall replace
the restriction that the (rectifiable) paths in the family be Jordan paths by the weaker condition that they be
admissible.
We discuss the implications of this below.
However, we also restrict attention to the case where the union of the images of the given (rectifiable) paths is dense in $X$;
this was not included in the definition of the term \lq useful' in \cite{Bland-Fein}. Note that this is only possible for semi-rectifiable, compact plane sets $X$.

\begin{definition}
Let $X$ be a compact plane set, and let $\paths$ be a family of paths in $X$.
Then $\paths$ is \textit{effective} if each path \jchange in $\paths$ is admissible, if each subpath of a path in $\paths$ belongs to $\paths$, and if
the union of the images of the paths in $\paths$ is dense in $X$.
\end{definition}

The next definition is as in \cite{Bland-Fein}.

\begin{definition}
Let $X$ be a compact plane set, let $\paths$ be a family of rectifiable
paths in $X$, and let $f,g \in C(X)$.
Then $g$ is an \textit{$\paths$-derivative}
of $f$
if, for all $\gamma \in \paths$,
we have \[\int_{\gamma} g = f(\gamma^+)-f(\gamma^-).\]
We define
\[ \BF{X} = \{ f \in C(X) : f \mbox{ has an $\paths$-derivative in
$C(X)$} \}.\]
\end{definition}

We are interested mostly in the case where $\paths$ is effective.
In this case it is easy to show that $\paths$-derivatives are unique.

We now revisit the theory of $\paths$-derivatives, as introduced in \cite{Bland-Fein}.
It was relatively
easy to prove the product rule for
$\paths$-derivatives  \cite[Theorem 4.9]{Bland-Fein}  when $\paths$ is useful by
using polynomial approximation on each of the Jordan arcs involved.
We shall show  in Theorem \ref{product-rule} below  that the product rule remains valid when $\paths$ is effective.

First we give a lemma concerning $\DD(X)$, based on
part of \cite[Theorem 4.17]{Bland-Fein}.

\begin{lemma}
\label{DDsubBF}
Let $X$ be a perfect, compact plane set, and suppose that $\paths$
is a set of
rectifiable paths in $X$.
\begin{enumerate}
\item[(i)]
Let $f \in \DD(X)$. Then the usual derivative $f'$ is also an $\paths$-derivative
of~$f$.
\smallskip
\item[(ii)]
Let $f_1, f_2 \in \DD(X)$. Then $f_1 f_2' + f_1' f_2$ is an $\paths$-derivative
of $f_1 f_2$.
\end{enumerate}
\end{lemma}
\begin{proof}
Since the product rule is valid for $\stddiff{X}$,
the result follows from the fund\-amental theorem of calculus for
rectifiable paths \cite[Theorem 3.3]{Bland-Fein}. \hfill
\end{proof}
\sskip
Note that the above  lemma applies, in particular, to functions in $R_0(X)$.
Using this, together with rational approximation and the method of repeated bisection,
we can now prove our new version of the product rule for $\paths$-derivatives.\smallskip

\begin{theorem}
\label{product-rule}
Let $X$ be a \jchange semi-rectifiable, compact plane set, and suppose that $\paths$ is an effective family of paths in $X$.
Let $f_1, f_2 \in \BF{X}$ have $\paths$-derivatives $g_1$, $g_2$, respectively.
Then $f_1 g_2 + g_1 f_2$ is an $\paths$-derivative of $f_1 f_2$.
\end{theorem}
\begin{proof}
Set $h=f_1 g_2 + g_1 f_2$ and $H=f_1 f_2$.
Assume, for contradiction, that the result is false.
Then there is a path $\gamma \in \paths$ such that
\[
\int_{\gamma} h \neq H(\gamma^+)-H(\gamma^-)\,.
\]
It is clear that we may suppose that $X=\Jim{\gamma}$ and that
$\paths$ consists of $\gamma$ and all its subpaths.
Since $\gamma$ is rectifiable, the area of $X$ is $0$.
By the Hartogs--Rosenthal theorem \cite[Corollary II.8.4]{Gamelin}, $R(X)=C(X)$.
Moreover, by various types of scaling, we may suppose that $|\gamma|= 1$
and that each of the four functions $f_1$, $g_1$, $f_2$, and $g_2$ have uniform norm at most $1/16$.
Set
\[
C=\left|\int_\gamma h - \left( H(\gamma^+) - H(\gamma^-) \right)\right | > 0\,,
\]
and then set $\varepsilon=\min\{C/2,1/16\}$.
For $j \in \{1,2\}$, choose a rational function $r_j$ with poles off $X$ such that
$|r_j-g_j|_X < \varepsilon$.

By an obvious repeated bisection method, we find nested decreasing subpaths $\gamma_n$ of $\gamma$ with
$|\gamma_n|=2^{-n}$ and such that
\[
\left|\int_{\gamma_n} h - \left ( H(\gamma_n^+) - H(\gamma_n^-) \right)\right | \geq \frac{C}{2^n}\quad (n \in \N)\,.
\]
There exists a (unique) point $a$ in $\bigcap_{n \in \N} \Jim{\gamma_n}$.
Choose an open disk $D$ centred on $a$ on which both $r_1$ and $r_2$ are analytic.
For each $j = 1,2$, choose an analytic anti-derivative $R_j$ of $r_j$ on $D$ such that $R_j(a)=f_j(a)$, and then
choose $n\in\N$ large enough that $\Jim{\gamma_n}$ is contained in $D$. Then it follows easily that
\[
|R_j-f_j|_{\Jim{\gamma_n}} < \frac{\varepsilon}{2^n}\quad (j = 1,2)\,.
\]
Set $r=R_1 r_2 + r_1 R_2$ and $R=R_1 R_2$.
By the preceding lemma,
\[
\int_{\gamma_n} r = R(\gamma_n^+) - R(\gamma_n^-)\,.
\]
However, easy calculations show that $|R-H|_{\Jim{\gamma_n}} < \varepsilon/2^{n+2}$ and $|r-h|_{\Jim{\gamma_n}} < \varepsilon/2$, from which we see that
\[
\left | \int_{\gamma_n} h -\int_{\gamma_n} r \right | < \frac{\eps}{2^{n+1}}
\]
and
\[
\left | ( H(\gamma_n^+) - H(\gamma_n^-) - ( R(\gamma_n^+) - R(\gamma_n^-))\right | < \frac{\eps}{2^{n+1}}\,.
\]
This quickly leads to a contradiction of the choice of $\gamma_n$.

Hence the result follows.
\end{proof}

\smallskip

Let $X$ be a semi-rectifiable, compact plane set, and suppose that $\paths$ is an effective family of paths in $X$. As observed above,
\jchange $\paths$-derivatives are unique in this setting,
and so we may denote the
$\paths$-derivative of a function $f \in \BF{X}$ by $f'$.
For  $f \in \BF{X}$, set
\[
\|f\|=|f|_X + |f'|_X\,.
\]

\begin{theorem}
\label{BF-complete}
Let $X$ be a semi-rectifiable, \jchange compact plane set, and suppose that $\paths$ is an effective family of paths in $X$.
Then $(\BF{X},\Jnorm)$ is a Banach function algebra on $X$ containing $\DD(X)$ as a subalgebra.
\end{theorem}
\begin{proof}
Let $f_1,f_2 \in \BF{X}$. By Theorem \ref{product-rule}, we have $f_1 f_2 \in \BF{X}$ and
$(f_1 f_2)' = f_1 f_2' + f_1' f_2$. It follows immediately that
$\|f_1 f_2\| \leq \|f_1\| \|f_2\|$, and so $(\BF{X},\Jnorm)$ is a normed algebra.

Let $(f_n)$ be a Cauchy sequence in $(\BF{X},\Jnorm)$. Then $(f_n)$ and $(f_n')$ are Cauchy sequences in $(C(X),|\cdot|_X)$, and so they converge uniformly on $X$, say $f_n \to f$ and $f_n' \to g$ as $n \to \infty$.

For each $\gamma \in \paths$, we have
\[
\int_\gamma g = \lim_{n \to \infty} \int_\gamma f_n' =
\lim_n(f_n(\gamma^+)-f_n(\gamma^-)) = f(\gamma^+) - f(\gamma^-)\,,
\]
and so $f \in \BF{X}$ with $f'=g$. Thus $(\BF{X},\Jnorm)$ is complete.

Finally, it follows from Lemma \ref{DDsubBF}(i) that $\DD(X) \subseteq \BF{X}$.
\end{proof}
\sskip
Note that the inclusion map from $\DD(X)$ to $\BF{X}$ is obviously isometric here.

Easy examples,
 such as \cite[Example 5.2]{Bland-Fein},
 show that  $\BF{X}$ need not be contained in $A(X)$ in this setting.

We do not know whether or not $\BF{X}$ is natural whenever $\paths$ is effective.
However, as for $\Aone(X)$ above, $\BF{X}$ is natural on $X$ if and only if the uniform closure of $\BF{X}$ is natural on $X$.

\section{The completion of $\DDXN$}
\noindent
In this section we shall discuss the completion, $\wdiff{X}$, of $\DDXN$.

We first give an example to show that $\wdiff{X}$ need not be semisimple.

 Recall that, for each  compact  space $X$, the semi-direct product $C(X) \ltimes C(X)$ is a Banach algebra for the product given by
\[
(f_1,f_2)(g_1,g_2 ) = (f_1g_1, f_1g_2 + f_2g_1)\quad (f_1,f_2,g_1,g_2 \in C(X))
\]
 and the norm given by
 \[
 \lV (f,g)\rV = \lv f \rv_X + \lv g \rv_X \;\,(f,g \in C(X))\,.
 \]
 The radical of this algebra is $\{0\} \ltimes C(X)$, a nilpotent ideal of index $2$. In particular, $C(X) \ltimes C(X)$ is not semisimple.
\smallskip

\begin{proposition}
Let $X$ be a perfect, compact plane set. Then the map
\[
\iota:f \mapsto (f,f')\,,\quad  \DD(X)\to  C(X) \ltimes C(X)\,,
\]
is an isometric algebra embedding, and $\wdiff{X}$
may be identified with the closure in $C(X) \ltimes C(X)$ of $\iota(\DD(X))$.
\end{proposition}

\begin{proof}
This is immediate; it was noted in \cite[pp. 53--54]{Dales-Thesis}. \end{proof}
\smallskip

\begin{example} \cite[Example 3.1.10(ii)]{Dales-Thesis}
In \cite{Bishop}, Bishop gave an example of a Jordan arc $J$ in the plane
with the property that the image under the above embedding
$\iota$ of the set of polynomial functions
is dense in  $C(J) \ltimes C(J)$.
It follows immediately from this that $\wdiff{J}$ is equal to
$C(J) \ltimes C(J)$. As observed above, this algebra is not semisimple.
\QED
\end{example}

There are no non-constant, rectifiable paths in Bishop's arc $J$. This is not too surprising
in view of our next result.\smallskip

\begin{theorem}
\label{semisimple-completion}
Let $X$ be a semi-rectifiable, compact plane set. Then $\wdiff{X}$ is semisimple.
\end{theorem}

\begin{proof}
Let $\paths$ be the family of all admissible paths in $X$. Then $\paths$ is effective.
By Theorem \ref{BF-complete}, $\BF{X}$ is a Banach function algebra, and we can regard $\wdiff{X}$
as the closure of $\DD(X)$ in $\BF{X}$.
From this, it is immediate that $\wdiff{X}$ is semisimple.
\end{proof}
\sskip

We do not know whether or not $\wdiff{X}$ is always equal to $\BF{X}$
under the conditions of Theorem \ref{semisimple-completion}, with $\paths$ equal to the family of all admissible paths in $X$.
Nor do we know, in this setting, whether or not the completeness of $\DDXN$ implies that $\DD(X)=\BF{X}$.
However, if we assume merely that $\paths$ is effective, then \cite[Example 5.2]{Bland-Fein} shows that it is possible for $\DD(X)$ to be complete without having $\DD(X)=\BF{X}$.
Nevertheless, in the next section we shall see that it can be useful to work with effective families $\paths$ which are somewhat smaller than the family of
all admissible paths in $X$.

\section{$\paths$-regularity}
\noindent
Several results concerning the relationship between $\DD(X)$ and $\BF{X}$ were obtained in \cite{Bland-Fein} under a variety of conditions on $\paths$.
We wish to simplify the conditions considered, while at the same time strengthening these results.
To do this we introduce the notion of $\paths$-regularity.

\begin{definition}
Let $X$ be a \jchange semi-rectifiable, compact plane set, and suppose that $\paths$ is an effective family of paths in $X$.
Then $X$ is \textit{$\paths$-regular at} a
point $z \in X$ if there is a constant $k_{z} > 0$ such that,
for every $w \in X$, there is a path $\gamma \in \paths$
joining $z$ to $w$ with $\Length{\gamma} \leq k_{z}|z-w|$.

The compact plane set $X$ is \textit{$\paths$-regular} if
$X$ is $\paths$-regular at every point of $X$.
\end{definition}

It is clear that $\paths$-regularity \jchange implies pointwise regularity.
The main gain in the following result over \cite[Theorem 5.1]{Bland-Fein} is  that the earlier requirement that
$\paths$ should `include all short paths' may now be removed.
\smallskip

\begin{theorem}
\label{F-reg-equality}
Let $X$ be a semi-rectifiable, \jchange compact plane set, and suppose that $\paths$ is an effective family of paths in $X$ such that
$X$ is $\paths$-regular. Then
\[\stddiff{X} = \BF{X}\,.\]
\end{theorem}

\begin{proof}
The proof is essentially the same as that of \cite[Theorem 5.1]{Bland-Fein}, with some very minor modifications.
\end{proof}

\begin{remark}
Since $\paths$-regularity implies pointwise regularity, we already know under
these conditions that
$\DDXN$ is complete. The content of Theorem \ref{F-reg-equality} is that we actually have equality
between $\stddiff{X}$ and $\BF{X}$.
\end{remark}

We now wish to investigate the connections between $\stddiff{X}$, $\Aone(X)$ (as discussed earlier),
and $\BF{X}$
in the case where $\Tint X$ is dense in $X$ and $\paths$ is an appropriate effective family of paths in $X$.
First we look more closely at $\paths$-differentiation and subpaths of a path.

\begin{lemma}
\label{maximal-intervals}
Let $\Gamma:[a,b]\to\C$ be a rectifiable path, and let
$f,g\in C(\Gamma)$. Let $\sI$ be the set of all non-degenerate closed subintervals $J$ of $[a,b]$
such that, for every subpath $\gamma$ of $\Gamma|_J$, we have
\[\int_\gamma g = f(\gamma^+) - f(\gamma^-)\,.\]
Set $E=\bigcup \sI$ {\rm (}the union of all the intervals in $\sI${\rm )}.
Then $\sI$ and $E$ have the following properties.
\begin{enumerate}
\item[(i)]
Let $I$ and $J$ be in $\sI$ with $I \cap J \neq \emptyset$. Then $I \cup J \in \sI$.
\smallskip
\item[(ii)]
Every $J \in \sI$ is contained in a unique maximal element of $\sI$ {\rm (}with respect to set inclusion{\rm )}.
\smallskip
\item[(iii)]
The maximal elements of $\sI$ partition
$E$.
\smallskip
\item[(iv)]
Taking interior relative to $[a,b]$, the subset
$[a,b]\setminus \Tint E$ has no isolated points, and hence is perfect whenever $E \neq [a,b]$.
\smallskip
\item[(v)]
Either $E = [a,b]$ or $[a,b]\setminus E$ is uncountable.
\end{enumerate}
\end{lemma}

\begin{proof} (i) This is essentially immediate from the definitions.\smallskip

(ii) Take $J \in \sI$, and let $c \in \Tint J$.
Set $b' = \sup \{d \in [a,b]: [c,d] \in \sI\}$ and
$a' = \inf \{d \in [a,b]: [d,c] \in \sI\}$. Then
it is easy to see that $a'< c <b'$ and that $[a',c]$ and $[c,b']$ are both in $\sI$,
whence (by (i)) $[a',b'] \in \sI$.
It is now clear that $[a',b']$ is the desired maximal element of $\sI$.\smallskip

(iii) This is clear from (i) and (ii).\smallskip

(iv) This follows quickly from the well-known fact that, if an open interval in $\R$ is contained in a countable disjoint union of compact intervals, then it must be entirely contained in one of the compact intervals.\smallskip

(v) Assume, for contradiction, that $[a,b]\setminus E$ is countable and non-empty.
Then, because of the nature of $E$, the set $[a,b]\setminus \Tint E$ would also be countable.
This contradicts the fact that no non-empty, countable, compact subset of $\R$ can be perfect.
The result follows.
\end{proof}
\sskip
The following example shows that $E$ may be dense in $[a,b]$, and yet not equal to $[a,b]$.
Recall that we denote the unit interval $[0,1]$ by $\I$.
\begin{example}
Consider the \lq identity path' $\Gamma:\I\rightarrow \C$, i.e.,
\[
\Gamma(t)=t \quad (t \in \I)\,.
\]
Let $f$ be the standard Cantor function on $\I$, and let $g$ be the zero function on $\I$, so that $f,g \in C(\I)$.
Then it is easy to see that the set $E$ from the preceding lemma
is simply the union of the closures of the complementary open intervals
of the Cantor set. In particular, $E$ is dense in, but not equal to, $\I$.
The complement of $E$ is uncountable, nowhere dense, and has Lebesgue measure $0$. However, by modifying
the Cantor set in a standard way, we may arrange that the complement of $E$ be uncountable, nowhere dense, and
have positive Lebesgue measure instead.
\QED
\end{example}

In \cite[Lemma 5.3]{Bland-Fein} it was shown that, in the case where $\paths$ is the set of all
rectifiable Jordan paths
in a compact plane set $X$,
$\BF{X}\subseteq A(X)$. However, investigation of
the proof
reveals that it is only necessary for $\paths$ to include sufficiently many admissible paths in
$\Tint X$. For example,
by Theorem \ref{F-reg-equality}, it is enough if, for every compact disc $D \subset \Tint X$,
$D$ is $\paths_D$-regular, where $\paths_D$ denotes the set of paths in $\paths$ whose images are contained in $D$.
Moreover, in this case, $\paths$-derivatives of functions in $\BF{X}$ agree with the usual derivatives of these functions on $\Tint X$.

We are now ready to establish a result connecting the spaces $\stddiff{X}$, $\Aone(X)$, and  $\BF{X}$.

\begin{theorem}
\label{mostly-interior}
Let $X$ be a compact plane set such that $\Tint X$ is dense in $X$.
Let $\paths$ be the set of all admissible paths $\gamma$ in $X$ such that
the complement of $\gamma^{-1}(\Tint X)$ is countable.
Then $\BF{X}$ is equal to $\Aone(X)$.
If $X$ is $\paths$-regular,
then both of these spaces are equal to $\stddiff{X}$.
\end{theorem}

\begin{proof}
Set $U=\Tint X$.
Since $\paths$ includes all admissible paths which are contained in $U$, the
remarks above show that $\BF{X} \subseteq A(X)$ and that $\paths$-derivatives of functions in $\BF{X}$ agree with the usual derivatives of these functions on $U$. Since $U$ is dense in $X$, it follows that $\BF{X} \subseteq \Aone(X)$.

Now let $f \in \Aone(X)$, and let $g$ be the (unique) continuous extension to $X$ of $(f|_U)'$.
It follows easily from Lemma \ref{maximal-intervals} that $g$ is, in fact, an $\paths$-derivative of $f$.
This gives $\Aone(X) \subseteq \BF{X}$. The equality $\BF{X}=\Aone(X)$  follows.

Since $\paths$ is effective, the rest is an immediate consequence of Theorem \ref{F-reg-equality}.
\end{proof}
\sskip
It is easy to extend the last part of this result \jchange to cover finite unions of sets satisfying appropriate regularity conditions.
\sskip
\section{The algebra $\Aone(X)$}
\noindent
We shall now see that, even if $X$ is uniformly regular (and hence geodesically bounded) and has
dense interior, it need not be the case that $\DD(X)=\Aone(X)$.

\begin{theorem}
\label{Cantor-squares}
There exists a uniformly regular, polynomially convex, compact plane set such that
$\Tint X$ is dense in $X$, and yet $\wdiff{X}=\DD(X)\neq \Aone(X)$.
In particular, $R_0(X)$ is not dense in $(\Aone(X),\Jnorm)$.
\end{theorem}
\begin{proof}
An example, based on the Cantor set, is shown in Figure \ref{cantor-example}, below.

\begin{figure}[htb]
\Jpdf{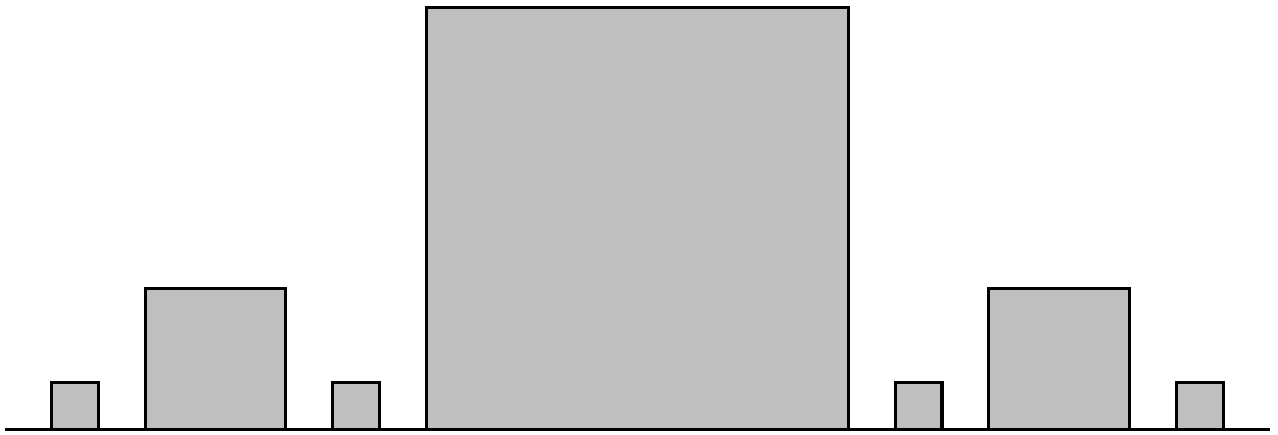}
\caption{The Cantor set with squares attached}
\label{cantor-example}
\end{figure}

Let $(I_n)_{n=1}^\infty$ be an enumeration of the closures of the complementary open intervals in the
standard Cantor middle thirds set, and say $I_n$ has length $l_n$.
Set
\[
   X = [0,1] \cup \bigcup_{n=1}^\infty \{x+\i y : x \in I_n,\, y \in [0,l_n] \}\,,
\]
i.e., $X$ is the set obtained by attaching to the unit interval
a closed square with base equal to $I_n$, for all $n \in \N$.
Then $X$ is easily seen to be a uniformly regular, compact plane set, and so $\DDXN$ is complete.

Let $g:[0,1]\rightarrow[0,1]$ be the usual Cantor function, and define
$f \in C(X)$ by
\[
f(x+ \i y)=g(x)\quad  (x+\i y \in X )\,.
\]
Then $f$ is locally constant on $\Tint X$, and so it is clear that $f \in \Aone(X)$.
However, $f$ is not differentiable on $[0,1]$, and so
$f \in \Aone(X) \setminus \DD(X)$.

Since $\DD(X)$ is a proper closed subalgebra of $\Aone(X)$, it follows immediately that $R_0(X)$ is not dense in $(\Aone(X),\Jnorm)$.
\end{proof}
\sskip
We do not know whether or not there is such an example with the additional property that the set $\Tint X$ is connected.

Theorem \ref{Cantor-squares} should be compared with the following polynomial approximation result, which is \cite[Theorem 5.8]{Bland-Fein} (see also \cite[Proposition 3.2.4]{Dales-Thesis}).
\sskip
\begin{theorem}
Let $X$ be a polynomially convex, geodesically bounded, compact plane set, and
let $\paths$ be the set of all admissible paths in $X$. Then $P_0(X)$ is dense in the Banach algebra
$(\BF{X},\Jnorm)$.\QED
\end{theorem}
Let $X$ be a semi-rectifiable, compact plane set, and let $\paths$ be the set of all admissible paths in $X$.
Obviously, if $P_0(X)$ is dense in $(\BF{X},\Jnorm)$, then
$P_0(X)$ is also dense in $(\DD(X),\norm)$. Thus
$P_0(X)$ is dense in $\DD(X)$ for every poly\-nomially convex, geodesically bounded, compact plane set $X$.
This last fact is also given by the argument in \cite[Proposition 3.2.4]{Dales-Thesis}, where it is noted that, for each pointwise regular, compact plane set $X$, $P_0(X)$ is dense in $\DDXN$ if and only if $X$ is polynomially convex. We do not know whether or not $P_0(X)$ is dense in $\DDXN$ for every polynomially convex, perfect, compact plane set $X$.

Our final result of this section shows that, even for uniformly regular, compact plane sets $X$ with dense interior, $\Aone(X)$ need not be dense in the uniform algebra $A(X)$.

\begin{theorem}  There is a compact plane set $X_0$ with the following properties:\smallskip

{\rm (i)} $X_0$ is uniformly regular, $\Tint X_0$ is connected and simply connected, and $\Tint X_0$ is dense in $X_0$;\smallskip

{\rm (ii)} $\widetilde{D}^{(1)}(X_0) =  {D}^{(1)}(X_0) = \Aone(X_0) \subseteq R(X_0)\/$;\smallskip

 {\rm (iii)} $\Aone(X_0) $ is not dense in $(A(X_0),\lv\,\cdot\,\rv_{X_0})$.
\end{theorem}

\begin{proof}   We consider the compact plane sets $X$  that are constructed in \cite[Chapter VIII, \S9]{Gamelin} to show that we may have $R(X)\subsetneq A(X)$.

The sets $X$ have the following form. Let  $(\Delta_n : n\in\N)$ be a sequence of open discs in the closed unit disc $\overline{\D}$ such that the family $\left\{\overline{\Delta}_n : n\in\N\right\}$  of closed discs is pairwise disjoint and $\sum_{n=1}^\infty r_n < \infty$, where $r_n$ is the radius of $\Delta_n$ for each $n\in \N$.  Then
\[
X = \overline{\D}\setminus \bigcup\{\Delta_n : n\in\N\}\,,
\]
and so $X$ is a compact plane set.

We \textit{claim} that each such set $X$ is uniformly regular.  To see this, take $z,w \in X$ with $z\neq w$, and first join $z$ and $w$ by a straight line $\ell$ in $\C$ of length $\lv z-w\rv $. Suppose  that $\ell\cap \Delta _n \neq \emptyset$ for some $n\in \N$.  Then the straight line $\ell\cap \Delta_n$, of length $t_n$, say, is replaced by an arc in the frontier of $\Delta_n $ of length at most $\pi t_n$.  It is not hard to show that  we obtain a path $\gamma$ in $X$ from $z$ to $w$  such that $\lv \gamma \rv \leq \pi \lv z-w\rv$, and so $\delta(z,w) \leq \pi \lv z-w\rv$. Thus $X$ is uniformly regular.

Let $J$ be a Jordan arc in $\overline \D$ such that $J$ has positive area and such that $J$ meets the unit circle in exactly one point.
The sets $\Delta_n$ may then be chosen so that the sequence of sets $(\Delta_n)$ accumulates precisely on $J$, in the sense that \jnew
\[
\bigcap_{n=1}^\infty \overline{\bigcup_{k=n}^\infty \Delta_k} = J\,.
\]
 It is shown in \cite[Chapter VIII, Example 9.2]{Gamelin} that this can be done in such a way that
each closed disc $\overline{\Delta}_n$ meets $J$ in exactly one point, say $z_n$: we then have that
$\Tint X$ is connected and simply connected, and that $\Tint X$ is dense in $X$. For  any set $X$ of this form, we have $R(X) \neq A(X)$.

Suppose that we have chosen discs $\Delta_n$ in this way.
For each $n\in\N$, choose  a new open disc $D_n$  with
\[
z_n \in \overline{D}_n  \subseteq \Delta_n \cup\{z_n\}\,,
\]
so that $D_n$ and $\Delta_n$ osculate at $z_n$.
Set
\[
X_0 = \overline{\D}\setminus \bigcup\{D_n : n\in\N\}\,,
\]
so that, for each $n \in \N$, the boundary circle $\partial \Delta_n$
is a subset of $\Tint X_0\cup\{z_n\}$.

It is clear that $X_0$ also satisfies clause (i) and that we still have $R(X_0) \subsetneq A(X_0)$. Since $X_0$ is uniformly regular, it follows from Propositions \ref{pr-sufficient} and \ref{sub-R(X)} that
\[
R_0(X_0) \subseteq \widetilde{D}^{(1)}(X_0) =  {D}^{(1)}(X_0)   \subseteq R(X_0)\,.
\]

 Let $z,w \in X_0$, and again join $z$ and $w$ by a straight line $\ell$ in $\C$ of length $\lv z-w\rv $.  Suppose  that $\ell\cap \Delta _n \neq \emptyset$ for some $n\in \N$.  Then we see by geometrical considerations that the straight line $\ell\cap \overline{\Delta}_n$, of length $t_n$, say, may be  replaced by a  path in $(\Delta_n\setminus \overline{D}_n)\cup\{z,w, z_n\}$ of length at most $(\pi +1)t_n$, and so  $z$ and $w$ can be joined by an admissible path $\gamma$ with $\lv \gamma \rv \leq (\pi +1)\lv z-w\rv$ such  that $\gamma$ is contained in
\[
\Tint X_0 \cup\{z,w\} \cup\{z_n : n\in \N\}\,,
\]
and such that the complement of  $\gamma^{-1}(\Tint X_0)$ is a countable set.

  Let $\mathcal F$ be the family specified in Theorem \ref{mostly-interior}. Then we have shown that $X_0$ is $\mathcal F$-regular, and so ${D}^{(1)}(X_0) = \Aone(X_0)$, giving (ii).

It follows from (ii) that the uniform closure of $\Aone(X_0) $  is $R(X_0)$. Clause (iii) follows because $R(X_0) \neq A(X_0)$.
\end{proof}

\section{Completeness of $\DD(X)$}
\noindent
We now return to the question of the completeness of $\DDXN$.

The following result is \cite[Proposition 3.1.4]{Dales-Thesis}; it was rediscovered by Honary and Mahyar in \cite{Honary-Mahyar1}.
\begin{theorem}
\label{D1-complete}
Let $X$ be a perfect, compact plane set. Then
$\DDXN$  is complete if and only if, for each $z \in X$, there exists $A_z > 0$
such that, for all $f \in \DD(X)$ and all $w \in X $, we have
\begin{equation}
|f(z)-f(w)|\leq A_z (|f|_X + |f'|_X)|z-w|\,.\label{D1-condition}
\end{equation}
\end{theorem}
\lift\QED
\smallskip
Note that $X$ need not be connected here
(However, the condition implies that $X$ has only finitely many components.)

For pointwise regular $X$, this condition is
certainly satisfied, and indeed the $|f|_X$ term may be omitted from the right-hand side
of (\ref{D1-condition}).
We now show that this $|f|_X$ term may also be omitted under the weaker assumption that $X$ be connected. First we require a lemma concerning functions whose derivatives are constantly $0$.
\smallskip
\begin{lemma}
\label{zero-derivative}
Let $X$ be a connected, compact plane set for which $(\DD(X), \norm)$ is complete.
Take $f \in \DD(X)$  such that $f'=0$. Then $f$ is a constant.
\end{lemma}

\begin{proof}   Assume towards a contradiction that there exists $f \in \DD(X)$  such that $f'=0$ and such that   $f$ is not a constant;  we can suppose that $\lv f \rv_X =1$ and that $1 \in f(X)$. By replacing $f$ by $(1+f)/2$ if necessary, we may also suppose that $1$ is the only value of modulus $1$ taken by $f$ on $X$. Set
\[Y = \{z\in X : \lv f(z)\rv =1\} = \{z\in X : f(z) =1\}\,.\]
Then $Y$ is a closed, non-empty subset of $X$,  and $Y \neq X$ because $f$ is not constant.  Since $X$ is connected, the subset $Y$ is not open in $X$, and so there exist $w_0 \in Y$ and  a sequence $(z_k)$ in $X\setminus Y$ such that $z_k \to w_0$ as $k\to\infty$.

For each $n \in \N$, set $g_n = 1-f^n$, so that $g_n \in D^{(1)}(X)$.  For each $n\in\N$, we have $g_n(w_0)=0$,
$g_n' = -nf^{n-1}f' =0$, and $\lv g_n\rv_X \leq 2$. By Theorem \ref{D1-complete}, there is a constant $C>0$ such that
\[
\lv g_n(z_k)\rv \leq 2C\lv z_k-w_0\rv\quad(n,k\in\N)\,.
\]
For each $k\in \N$, we have $\lim_{n\to\infty}g_n(z_k) =1$ because $\lv f(z_k)\rv < 1$, and so
\[
1 \leq 2C\lv z_k-w_0\rv\,\;(k\in\N)\,.
 \]
But $\lim_{k\to \infty}\lv z_k-w_0\rv =0$, and so this is the required contradiction.
\end{proof}
\smallskip
Note that this result fails without the assumption that $(\DD(X), \norm)$ is complete. For example, if $\gamma$ is a Jordan path which is a subpath of the famous Koch snowflake curve, then the function $f=\gamma^{-1}$ has derivative $0$ on the Jordan arc $\Jim{\gamma}$.
\smallskip
We are now ready to eliminate the $|f|_X$ term from the right-hand side
of equation (\ref{D1-condition}) under the assumption that $X$ is connected.

For convenience, we introduce the following notation.
Let $X$ be a perfect, compact plane set, and let $z_0 \in X$. Then we define
\[
\Mone=\{f \in \DD(X): f(z_0)=0\}\,,
\]
so that $\Mone$ is a maximal ideal in $\DDX$.
\begin{theorem}
 \label{term-eliminated-thm}
 Let $X$ be a connected, compact plane set  for which $(D^{(1)}(X), \norm)$ is complete.  Let $z_0 \in X$.  Then there exists a constant $C_1 >0$ such that, for all $f \in \Mone$, we have
 \[
 \lv f\rv_X \leq C_1\lv f'\rv_X\,.
 \]
Furthermore, there exists another constant $C_2 > 0$ such that, for all $f\in\DD(X)$  and all $w \in X$, we have
\begin{equation}\label{term-eliminated}
\lv f(z_0)-f(w)\rv \leq C_2\lv f'\rv_X\lv z_0-w\rv\,.
 \end{equation}
\end{theorem}

\begin{proof}   We shall
first prove the existence of the constant $C_1$.

Assume towards a contradiction that there is a sequence $(f_n)\in\Mone$  such that $\lv f_n\rv_X =1$ for each $n\in\N $, but such that  $\lv f'_n\rv_X \to 0$ as $n\to\infty$.  We can suppose that $\lv f'_n\rv_X \leq 1\,\;(n\in\N)$.  Set $\mathcal S = \{f_n: n\in\N\}$.

Let $z\in X$. By Theorem \ref{D1-complete}, there is a constant $C_z>0$ such that
\[
\lv f(z)- f(w)\rv \leq C_z(\lv f\rv_X + \lv f'\rv_X)\lv z-w\rv
\leq 2C_z \lv w-z\rv
\quad (f\in  \mathcal S, \,w\in X)\,.
\]

We {\it claim\/} that $\mathcal S$ is an equicontinuous family at $z$.  Indeed, take $\varepsilon >0$, and let $U =
\{w \in X : \lv w-z\rv <  \varepsilon/2C_z\}$, so that $U$ is a neighbourhood of $z  \in X$. For each $w \in U$, we have
$\lv f(w)- f(z)\rv < \varepsilon$ for each $f\in \mathcal S$, giving the claim.  Thus $\mathcal S$ is   equicontinuous on $X$.  Certainly
$\mathcal S$ is bounded in $(C(X), \lv \,\cdot\,\rv_X)$.

By Ascoli's theorem \cite[Theorem A.1.10]{Dales}, $\mathcal S$  is relatively compact in $(C(X), \lv \,\cdot\,\rv_X)$.  By passing to a subsequence, we may suppose that there exists $f\in C(X)$ such that $\lv f_n -f\rv_X \to 0$ as $n\to\infty$. Clearly we have $f(z_0) = 0$ and $\lv f \rv_X =1$.  We know that $\lv f'_n \rv_X \to 0$ as $n\to\infty$, and so $(f_n)$ is a Cauchy sequence in  $(D^{(1)}(X), \norm)$. Since  $(D^{(1)}(X), \norm)$ is complete, $(f_n)$ is convergent in this space. Clearly $\lim_{n\to \infty}f_n = f $ in $D^{(1)}(X)$, and so $f' =0$.

By Lemma \ref{zero-derivative}, $f$ is a constant. But $f(z_0) = 0$, and so $f=0$, a contradiction of the fact  that $\lv f \rv_X =1$.

This proves the existence of the desired constant $C_1$.

We now set $C_2 = A_{z_0}(1+C_1)$, where $A_{z_0}$ is the constant from (\ref{D1-condition}).
Equation (\ref{term-eliminated}) now follows. \end{proof}

\smallskip

Theorem \ref{term-eliminated-thm} does not hold in the absence of  either of the hypotheses that $X$ be connected or
that $(D^{(1)}(X), \norm)$ be complete.
\smallskip
The following corollary is now immediate.
\begin{corollary}
\label{term-eliminated-cor}
Let $X$ be a connected, compact plane set. Then $(D^{(1)}(X), \norm)$ is complete if and only if, for each $z \in X$, there exists $B_z>0$ such that,
for all $f\in\DD(X)$  and all $w \in X$, we have
\begin{equation}
\lv f(z) - f(w)\rv \leq B_z\lv f'\rv_X\lv z-w\rv\,.
 \end{equation}
\lift\QED
\end{corollary}
\smallskip
Let $X$ be a polynomially convex, geodesically bounded, compact plane set. Then $X$ is connected, and we know that $P_0(X)$ is dense in $\DDXN$. From this we immediately obtain the following further corollary.
\begin{corollary}
\label{poly-term-eliminated}
Let $X$ be a polynomially convex, geodesically bounded, compact plane set. Then $(D^{(1)}(X), \norm)$ is complete if and only if, for each $z \in X$, there exists $B_z>0$ such that,
for all $p\in P_0(X)$  and all $w \in X$, we have
\begin{equation}
\label{poly-condition}
\lv p(z) - p(w)\rv \leq B_z\lv p'\rv_X\lv z-w\rv\,.
 \end{equation}
 \lift\QED
\end{corollary}
\smallskip
We conclude this section with some polynomial approximation results. Of course, if we knew that $X$ were geodesically bounded whenever $X$ is a compact plane set for which $\DDXN$ is complete, then (ii) below would be immediate.

\begin{corollary} Let $X$ be a connected,  polynomially convex, compact plane set  for which $(D^{(1)}(X), \norm)$ is complete.\smallskip

{\rm (i)} Take $z_0\in X$. Then the map
\[
f\mapsto f'\,,\quad \Mone \to P(X)\,,
\]
is a  bicontinuous linear isomorphism.\smallskip

 {\rm (ii)} The algebra $P_0(X) $ is dense in $(D^{(1)}(X), \norm)$.\end{corollary}

\begin{proof} (i) The map
 \[
T: f\mapsto f',\quad \Mone \to P(X)\,,
\]
is    linear, and it  is clearly continuous. By Lemma \ref{zero-derivative}
(or Theorem \ref{term-eliminated-thm}), $T$ is injective.

Let $g \in P(X)$. Then there is a sequence $(q_n)$ of polynomials with $\lv q_n - g\rv_X \to 0$ as $n\to \infty$.  For each $n\in\N$, define a polynomial $p_n$ by requiring that $p_n' =q_n$ and $p_n(z_0) = 0$. By Theorem \ref{term-eliminated-thm}, $(p_n)$ is a Cauchy sequence in $(C(X), \lv \,\cdot\,\rv_X)$, and so there exists $f\in C(X)$ such that $\lv p_n - f\rv_X\to 0$ as $n\to \infty$.   Clearly  $(p_n)$ is also a Cauchy sequence in $(D^{(1)}(X), \norm)$, and so $f \in \DD(X)$ and $p_n \to f$ in $(D^{(1)}(X), \norm)$. Clearly $f \in \Mone$,
and $f' =Tf =g$. This shows that $T$ is a surjection.

The continuity of $T^{-1}$ is now immediate from either Theorem \ref{term-eliminated-thm} or the open mapping theorem.
\smallskip

(ii)  Let $f\in D^{(1)}(X)$. Then $f' \in A(X) = P(X)$, and so there is a sequence $(q_n)$ of polynomials such that $\lv q_n - f'\rv_X \to 0$ as $n\to \infty$.  Let $(p_n)$ be a sequence of polynomials such that  $p_n' =q_n$ and $p_n(z_0) = f(z_0)$  for $n\in\N$.
Since $T^{-1}$ is continuous, we have
\[p_n - f = T^{-1}(q_n-f') \to 0\]
 in $(D^{(1)}(X), \norm)$ as $n \to \infty$, giving (ii).
 \end{proof}
\smallskip

Now suppose that $X$ is  a connected, compact plane set (which need not be polynomially convex) such that $(D^{(1)}(X), \norm)$ is complete. Take $z_0\in X$, and define $\Mone$ and $T:\DD(X)\to A(X)$ as above, i.e.,
\[
T(f)=f'\quad(f \in \Mone)\,.
\]
Then the  above argument shows that $T(\Mone)$ is a closed linear subspace of $A(X)$.  For example, in the case where $X= \T$, the range of the map $T$ is the space of functions $g\in C(\T)$ such that $\int_\T g(z) {\dd}z =0$.

Let $X$ be a polynomially convex, perfect, compact plane set such that the normed algebra $(D^{(1)}(X), \norm)$ is complete.
Then $X$ has only finitely many components, and the above polynomial approximation result holds on each component separately. It follows easily that $O(X)$ is dense in $(\DD(X),\norm)$. However, as we observed earlier, for polynomially convex, perfect, compact plane sets $X$, $P_0(X)$ and $O(X)$ have the same closure in $(\DD(X),\norm)$. Thus we obtain the following further corollary.

\begin{corollary} Let $X$ be a polynomially convex, perfect, compact plane set for which $(D^{(1)}(X), \norm)$ is complete.
Then $P_0(X)$ is dense in $(D^{(1)}(X), \norm)$.\QED
\end{corollary}

 \section{Sufficient conditions  for the incompleteness of  $\DDXN$}
 \noindent
Let $X$ be a connected, compact plane set $X$.
We shall identify a variety of geometric conditions on $X$ which are sufficient for $\DDXN$ to be incomplete.
We begin by defining a function
$\QX:X \rightarrow [0,\infty]$ in order to make applications of  Theorem \ref{term-eliminated-thm} more efficient.
For $z \in X$, we set
\[
\QX(z) = \sup \left\{\frac{|f(z)-f(w)|}{|z-w|}:w \in X \setminus \{z\},\, f \in \DD(X),\, |f'|_X \leq 1\,\right\}\,.
\]
Thus, for $z_0 \in X$, condition (\ref{term-eliminated}) holds at $z_0$ if and only if $\QX(z_0)<\infty$, in which case we may take $C_2=\QX(z_0)$.

We may now rephrase Corollary \ref{term-eliminated-cor} in terms of $\QX$ as follows.

\begin{theorem}
Let $X$ be a connected, compact plane set. Then $\DDXN$  is complete if and only if, for all $z \in X$, we have $\QX(z)<\infty$.\QED
\end{theorem}
\smallskip
Let $X$ be a connected, compact plane set.
By Proposition \ref{pr-sufficient}, $\DDXN$ is complete whenever $X$ is pointwise regular. We shall now try to establish the converse of this result; we have no counter-example to the possibility that the converse always holds.
The following partial result eliminates a large class of compact plane sets, including all rectifiable Jordan arcs, as possible counter-examples.

\begin{theorem}
Let $X$ be a polynomially convex, rectifiably connected, compact plane set with empty interior. Suppose that $X$ is not pointwise regular. Then $(\DD(X),\norm)$ is incomplete.
\end{theorem}

\begin{proof}
Since $X$ is not pointwise regular, there exists $z_0 \in X$ such that $X$ is not regular at $z_0$.
We shall show that $\QX(z_0)=\infty$.

Let $C>0$. Then there exists $w \in X$ with $\gd(z_0,w)>C|z_0-w|$.
Let $\gamma$ be a geodesic from $z_0$ to $w$ in $X$.
It is easy to see, using the definition of arc length and Tietze's extension theorem, that there is a function
$f \in C(X)$ with $|f|_X<1$ and such that
\[
\left| \int_\gamma f \right| > C|z_0-w|\,.
\]
 As the polynomials are dense in $C(X)$, we may suppose that $f$ is, in fact, a polynomial.
 Let $p$ be the unique polynomial satisfying $p(z_0)=0$ and such that $p'=f$.
 Then $|p(z_0)-p(w)| > C |z_0-w|$, and $|p'|_X <1$. This shows that $\QX(z_0)>C$. As the choice of $C>0$ was arbitrary, $\QX(z_0)=\infty$, as claimed.

Thus $\DDXN$ is incomplete.
 \end{proof}\smallskip

Although the preceding proof does not apply, nevertheless it is true that the normed algebra $(\DD(J),\Jnorm)$ is incomplete whenever
$J$ is a non-rectifiable Jordan arc. The proof requires two lemmas.

\begin{lemma}Let $z_0,w_0 \in \C$ with $z_0 \neq w_0$, let $\gamma$ be a Jordan path from $z_0$ to $w_0$ in $\C$, and let $B \in \Rplus$.
Then there exists $f \in \DD(\Jim{\gamma})$ with $0<|f'|_{\Jim{\gamma}} \leq 3$, with
$f'(z_0)=f'(w_0)=0$, with $f(z_0)=B$, with $f(w_0)\in \Rplus$, and with
\[
f(w_0) > f(z_0)+\frac12|w_0-z_0|\,.
\]
\end{lemma}
\begin{proof}
Set $\eta=|z_0-w_0|/100$.
Let $z_1$ be the first point of $\gamma$ with $|z_1-z_0|=\eta$, and let $w_1$ be the last point of $\gamma$ with $|w_1-w_0|= \eta$. We may split $\gamma$ into three Jordan paths:
$\gamma_1$ from $z_0$ to $z_1$, $\gamma_2$ from $z_1$ to $w_1$, and $\gamma_3$ from
$w_1$ to $w_0$. Note that $|z-z_1|\leq 2\eta$ on $\gamma_1$, while
$|z-w_1| \leq 2 \eta$ on $\gamma_3$.

Define $g(z)=z-z_0$ for $z \in \gamma_2$. For $z$ in $\gamma_1$, define
\[g(z) = z-z_0 - \frac{(z-z_1)^2}{2(z_0-z_1)}\,.\]
For $z \in \gamma_3$, define
\[
g(z) = z - z_0 - \frac{(z-w_1)^2}{2(w_0-w_1)}\,.
\]
Finally, set $f(z) = B + \e^{\i \theta} (g(z) - g(z_0))$, where $\theta \in \R$ is chosen so that we have $\e^{\i \theta} (g(w_0) - g(z_0)) \in \R^+$.

It is then easy to check that $f$ has the desired properties.
\end{proof}

\begin{lemma}
Let $z_0,w_0 \in \C$, let $A>0$, and suppose that $\gamma$ is a Jordan path from $z_0$ to $w_0$ in $\C$ whose length {\rm (}which may be infinite{\rm )} is greater than $A$.
Then there exists $f \in \DD(\Jim{\gamma})$ with $0<|f'|_{\Jim{\gamma}} \leq 3$,
with $f'(z_0)=0$, with $f'(w_0)=0$, and with $|f(z_0)-f(w_0)|> A/2$.
\end{lemma}

\begin{proof}
Choose finitely many intermediate points on $\gamma$ strictly between $z_0$ and $w_0$, say $z_1, z_2, z_3, \dots, z_{n-1}$ in order along $\gamma$, such that, setting $z_n=w_0$, we have
\[\sum_{k=1}^n |z_k-z_{k-1}| > A\,.\]

We now apply the previous lemma successively to the $n$ subpaths joining $z_{k-1}$ to $z_k$ for $k=1,2,\dots,n$, starting with $B=0$ for the first arc, and arranging for the function values to match at the endpoints. We patch together the resulting functions in the obvious way to obtain the desired function $f$.
\end{proof}

\begin{theorem}
\label{arc-conjecture}
Let $J$ be a non-rectifiable Jordan arc. Then the normed algebra $(\DD(J),\Jnorm)$ is incomplete.
\end{theorem}

\begin{proof}
By the definition of non-rectifiable Jordan arc, we have $J=\Jim{\gamma}$ for some non-rectifiable Jordan path $\gamma:[a_0,b_0]\rightarrow \C$.
It follows easily by compactness and symmetry that we may suppose that there exists an $a\in [a_0,b_0)$ such that, for all $b \in (a,b_0]$,
the restriction $\gamma|_{[a,b]}$ is also non-rectifiable.

Set $z_0=\gamma(a)$.
We shall show that $\QX(z_0)=\infty$.

Let $C>0$.
Choose $b \in (a,b_0)$ such that $C |\gamma(a)-\gamma(b)|< 1$, and
set $w_0 = \gamma(b)$, so that
\[
\frac{1}{|z_0 - w_0|} > C\,.
\]
Set $A=6$, and
apply the previous lemma to the restriction $\gamma|_{[a,b]}$. Extend the resulting function $f$ to be constant on
$\gamma([a_0,a])$ and on  $\gamma([b,b_0])$. We then see that
$f \in \DD(J)$, that
$|f(z_0)-f(w_0)|> A/2 = 3$,
and that
\[
0<|f'|_J \leq 3 < |f(z_0)-f(w_0)|\,.
\]
But then

\[
\QX(z_0) \geq \frac{|f(z_0)-f(w_0)|}{|f'|_J |z_0-w_0|} > C\,.
\]
As the choice of $C>0$ was arbitrary, $\QX(z_0) = \infty$, and so $\DDXN$  is incomplete.
\end{proof}
\sskip
We do not know whether or not $\DDXN$ is incomplete for each poly\-nomially convex, compact plane set $X$ which has empty interior and which is not pointwise regular.

\smallskip
We shall now introduce some useful test functions which will enable us to give lower bounds for $\QX(z_0)$ whenever a compact plane set $X$ has a suitable \lq dent' near $z_0$.

We denote by $D_0$ the standard cut plane obtained by deleting the non-positive real axis from $\C$.
For $z \in D_0$,
we denote the principal argument of $z$ by $\Arg z$ (so that $- \pi < \Arg z < \pi$),
and we define
\[
\Log z = \log |z| + \i \Arg z\quad (z \in D_0)
\]
(so that $\Log$ is the {\sl principal logarithm} defined on $D_0$).
For $z \in D_0$ and $\alpha \in \C$, we define $z^\alpha$ by
$z^\alpha = \exp(\alpha \Log z)$, and we define the function
\[
Z^\alpha:\,z \mapsto z^\alpha,\quad D_0\rightarrow \C\,.
\]

We now give some elementary facts about the function $Z^\i$.

\begin{lemma}
Let $z\in D_0$. Then $|z^\i| = \e^{-\Arg z}$, and so  $\e^{-\pi} < |z^\i| < \e^{\pi}$. If $z$ is in the second quadrant,
we have $|z^\i| \leq \e^{-\pi/2}$, while if $z$ is in the third quadrant, then $|z^\i|\geq \e^{\pi/2}$.
\QED
\end{lemma}

\begin{corollary}
Let $z, w \in D_0$, and suppose that $z$ is in the second quadrant and that $w$ is in the third quadrant {\rm (}or vice versa{\rm )}. Then
\[
|z^\i - w^\i| \geq  \e^{\pi/2} -  \e^{-\pi/2} \geq 1\,.
\]
\lift\QED
\end{corollary}
\sskip
We shall be particularly interested in the properties of the analytic function $F=Z^{1+\i}$ on $D_0$.
Note that $F'(z)=(1+\i) z^\i~(z \in D_0)$.

\begin{lemma}
The following estimates hold for all $z \in D_0$:
\smallskip
\begin{enumerate}
\item[(i)] $\e^{-\pi} |z| \leq |F(z)| \leq \e^{\pi} |z|$;
\smallskip
\item[(ii)] $\sqrt 2 \,\e^{-\pi} \leq |F'(z)| \leq \sqrt 2 \,\e^{\pi}$.\QED
\end{enumerate}
\end{lemma}

\begin{lemma}
Let $z,w \in D_0$. Suppose that $z$ is in the
second quadrant and that $w$ is in the third quadrant.
Then
\[
|F(z)-F(w)| \geq |z| - \e^\pi |z-w|\,.
\]
\end{lemma}
\begin{proof}
Writing
\[
F(z)-F(w) = z^\i z - w^\i w = (z^\i - w^\i)z + w^\i(z-w)\,,
\]
we obtain (using our preceding estimates) the inequalities
\[
|F(z)-F(w)| \geq |z^\i -w^\i||z| - |w^\i||z-w| \geq |z| - \e^\pi |z-w|\,,
\]
as required.
\end{proof}
\sskip
\noindent
{\bf Note.} In this setting, if $|z-w|/|z|$ is very small, then $|F(z)-F(w)|$ is much bigger than $|z-w|$.

As an immediate consequence of the above lemmas, we obtain the following result concerning the function $\QX$ when $X \subseteq D_0$.

\begin{proposition}
There exist universal constants $C_Q, C_Q'>0$ satisfying the following.
Let $X$ be a connected, compact plane set with $X \subseteq D_0$.
Suppose that $z,w \in X$, and that $z$ is in the second quadrant and $w$ is in the third quadrant.
Then
\[
\QX(z) \geq C_Q \frac{|z|}{|z-w|} - C'_Q\,.
\]
\lift\QED
\end{proposition}

These same universal constants $C_Q, C'_Q$ will be fixed for the rest of this paper. We shall use them repeatedly below.

As $\QX$ is clearly invariant under reflections, rotations and translations, we obtain the following further corollary.

\begin{corollary}
\label{half-lines}
Let $a \in \C$, let $L$ be a closed half-line in $\C$ joining the point $a$ to $\infty$, and let $H$ be the closed half-plane containing $L$ and bisected by it. Let $X$ be a connected, compact plane set with $X\cap L = \emptyset$, and suppose that $z, w \in X \cap H$ are such that the straight line joining $z$ to $w$ meets $L$.
Then
\[
\QX(z) \geq C_Q \frac {|z-a|}{|z-w|} - C'_Q\,.
\]
\lift\QED
\end{corollary}

\smallskip

We shall now consider the case of a compact plane set with  a sequence of dents.

The context will be as follows.
Let $X$ be a connected, compact plane set.
Let $(w_n)$ be a convergent sequence in $X$ with limit $z_0$.
Let $(a_n)$ be a bounded sequence in $\C \setminus X$, and suppose that, for each $n \in \N$, there is a closed half-line $L_n \subseteq \C\setminus X$
joining $a_n$ to $\infty$.
Let $H_n$ be the closed half-plane containing $L_n$ and bisected by it.
Suppose further that, for each $n \in \N$,  $w_n$ and $z_0$ are in $X \cap H_n$ and the straight line joining $z_0$ to $w_n$ meets $L_n$.

\begin{theorem}
\label{long-dents}
With notation as above, suppose that
\[
|z_0-w_n|=o(|z_0-a_n|) \quad\text{as}\:\, n \to \infty\,.
\]
Then $\DDXN$ is incomplete.
\end{theorem}
\begin{proof}
We shall show that $Q_X(z_0)=\infty$.
By Corollary \ref{half-lines}, for each $n \in \N$, we have
\[
\QX(z_0) \geq C_Q \frac {|z_0-a_n|}{|z_0-w_n|} - C'_Q\,.
\]
Letting $n \to \infty$, we obtain $Q_X(z_0)=\infty$, and so $\DDXN$  is incomplete.
\end{proof}

We now describe some classes of sets to which  the above theorem applies.

We start from the closed unit \jchange square $S=\I \times \I$ in $\R^2$.
Let $(r_n)$ and $(s_n)$ be sequences in the open interval $(0,1)$ such that $(s_n)$ is strictly decreasing and converges to $0$.
For each $n \in \N$, set
\[
R_n = [0,r_n),\times (s_{2n},s_{2n-1})\,.
\]
Set $X=S \setminus \bigcup_{n=1}^\infty R_n$.
The set $X$, regarded as a subset of $\C$, is illustrated in
Figure~\ref{square-BF-example}, below.

\begin{figure}[htb]
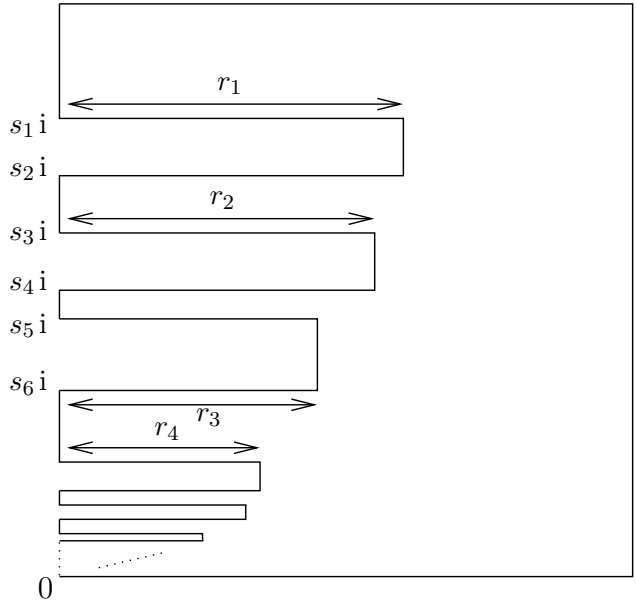
\caption{A square with rectangles deleted}
\label{square-BF-example}
\end{figure}

Clearly $X$ is a polynomially convex, geodesically bounded, compact plane set, $\Tint X$ is connected and dense in $X$, and $X$ is regular at all points of $X \setminus \{0\}$.
\sskip
With the assistance of Theorem \ref{long-dents}, we see easily that the following conditions on $X$ are equivalent:

\begin{enumerate}
\item[(a)] $X$ is regular at $0$\,;
\sskip
\item[(b)] $X$ is pointwise regular;
\sskip
\item[(c)] $\DDXN$ is complete;
\sskip
\item[(d)] $r_n=O(s_{2n-1})$.
\end{enumerate}

In particular, these sets provide many examples of compact plane sets $X$ such that $\Tint X$ is connected and dense in $X$, and yet $\DDXN$ is incomplete, thus answering a question raised in \cite{Bland-Fein}.

\sskip
It is fairly easy to generalize Theorem \ref{long-dents} somewhat by replacing the half-lines $L_n$ by suitable curves joining $a_n$ to $\infty$, and imposing appropriate conditions on the positions of $z_0$ and $w_n$ relative to these curves. All that is required is that suitable versions of Corollary \ref{half-lines} be valid for these curves, without losing control of the constants involved. (There are problems, for example, with sequences of curves which spiral round increasingly often.) Although such a generalization would allow us to prove the incompleteness of $\DDXN$ for some additional compact plane sets $X$, there are clearly limitations to this method, and we shall not pursue it further here.

\sskip
Let $X$ be a compact plane set. Recall that $X$ is \emph{star-shaped} with a \emph{star-centre} $a \in X$ if, for all
$z \in X$, the straight line joining $a$ to $z$ lies entirely within $X$. The compact plane set
$X$ is {\sl radially self-absorbing} if, for all $r>1$,
$X \subseteq \Tint(r X)$, where $rX = \{rz:z \in X\}$. Such radially self-absorbing sets are discussed in \cite{FLO}.

Radially self-absorbing sets are always star-shaped, polynomially convex and geodesically bounded, and have dense interior.

\sskip
Using Theorem \ref{long-dents}, we now construct a radially self-absorbing set $X$ such that $\DDXN$ is incomplete; this answers a question raised in \cite{Bland-Fein}.

\begin{example}
Set $z_0=1$. For each $n \in \N$, set
\[
r_n=\frac{1}{4\sqrt{n}}\,,
\quad\alpha_n = \frac{\pi}{4 n^2}\,,
\quad\beta_n = \frac{\alpha_n + \alpha_{n+1}}{2}\,,
\]
and
\[
w_n=\e^{\i \alpha_n}\,,
\quad z_n=(1-2r_n)\e^{\i \beta_n}\,,
\quad a_n=(1-r_n)\e^{\i \beta_n}\,.
\]
Note that $|z_0-w_n|=O(1/n^2)$, while
\[
|z_0-a_n|\geq\frac{1}{4\sqrt{n}}\quad(n \in \N)\,.
\]
Thus $|z_0-w_n|=o(|z_0-a_n|)$ as $n \to \infty$.

For each $n \in \N$, \jchange let $U_n$ be the acute open sector of $\C$ with vertex at $z_n$ and boundary lines passing through
$w_n$ and $w_{n+1}$.
Let $X$ be the compact set obtained by deleting from the closed unit disc
the union of all the open sectors $U_n$.
\begin{figure}[htb]
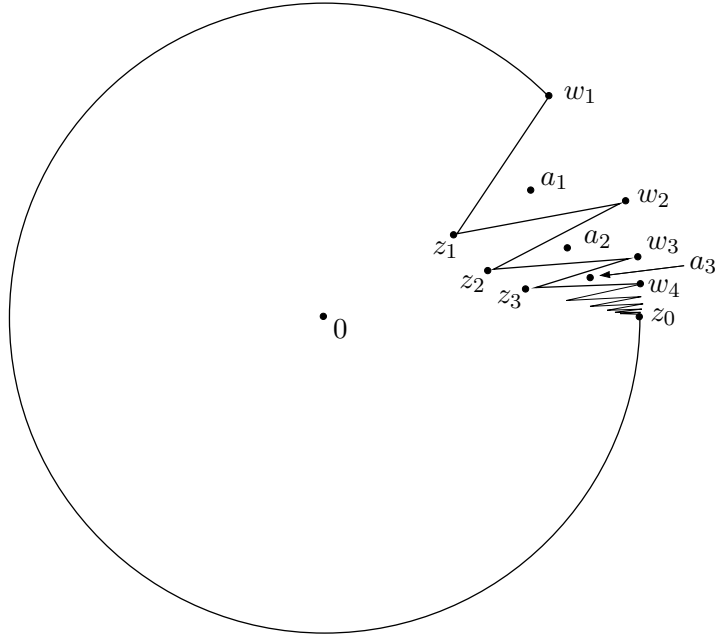
\caption{A \lq bad' radially self-absorbing set}
\label{RSA-example}
\end{figure}

The set $X$ is illustrated in Figure \ref{RSA-example}, above.
Clearly $X$ is radially self-absorbing and satisfies the conditions of Theorem \ref{long-dents}.

Thus
$\DDXN$ is incomplete. \QED
\end{example}

\sskip
We shall next prove our conjecture for star-shaped sets.

\begin{theorem}
Let $X$ be a star-shaped, compact plane set $X$. Then the normed algebra $\DDXN$ is complete if and only if $X$ is pointwise regular.
\end{theorem}
\begin{proof}
We need to consider only the case where $X$ is not pointwise regular, and to prove that $\DDXN$ is incomplete in this case.

Let $a$ be a star-centre for $X$. Given that $X$ is not pointwise regular, there must exist $z_0 \in X$ such that $X$ is not regular at $z_0$. Clearly $X$ is regular at $a$, and so $z_0 \neq a$.
Choose a sequence of points $(w_n) \in X$ converging to $z_0$ and such that, for all $n \in \N$,
$\delta(z_0,w_n) > 100n |z_0 - w_n|$.
Clearly, none of the points $w_n$ lie on the line segment joining $a$ to $z_0$. We may further suppose that
$|z_0-w_n| < |z_0-a|/100n$ for all $n \in \N$.
For fixed $n \in \N$, consider the following path $\gamma_n$ in $\C$ from $z_0$ to $w_n$: $\gamma_n$ starts at $z_0$, travels along a straight line towards $a$ for distance $3n|z_0-w_n|$, then travels through a small angle around a circle centred on $a$ until it meets the straight line joining $a$ to $w_n$, and finally travels along this straight line to $w_n$. It is clear that
$|\gamma_n|<\delta(z_0,w_n)$, and so the path $\gamma_n$ is not contained in $X$. Thus there must be a point $a_n$ on the short circular arc which is not in $X$. Let $L_n$ be the closed half-line joining $a_n$ to $\infty$ obtained as a continuation of the straight line from $a$ to $a_n$.

It is now easy to check that the conditions of Theorem \ref{long-dents} are satisfied, and so $\DDXN$ is incomplete.
\end{proof}

Note that most of the results above concern polynomially convex sets; we now consider sets which are not polynomially convex.

Let $X$ be a compact plane set such that $\DDXN$ is complete. Then an elementary argument shows that $(\DD(\widehat{X}),\norm)$ is also complete.

The converse is false. For example, let $(y_n)$ be a sequence in $(0,1)$ such that $(y_n)$ is strictly decreasing
\jchange and converges to $0$.
Set $A=\{0,1\}$, set
\[
B=\{0\} \cup \{y_n: n \in \N\}\,,
\]
and set
\[
X=\left(A \times [0,1]\right) \cup \left([0,1] \times B\right)\,,
\]
so that $X$ is the union of the boundary of a square with a suitable sequence of line segments running across it.
The set $X$, regarded as a subset of $\C$, is shown in Figure \ref{square-crossed}, below.

\begin{figure}[htb]
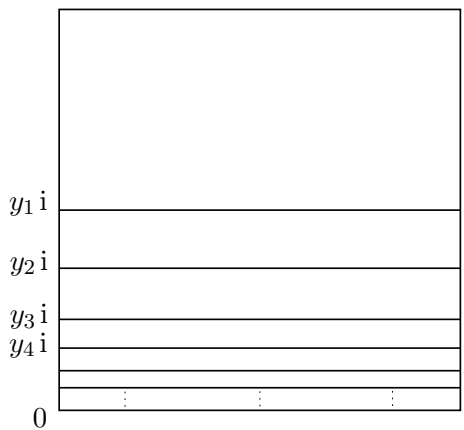
\caption{A square crossed by lines}
\label{square-crossed}
\end{figure}

Although none of our theorems apply, it is not hard to show directly that $\DDXN$ is incomplete. However,
$\widehat X$ (which is a square) is convex, and hence is uniformly regular. Thus $(\DD(\widehat X),\norm)$ is complete.

\sskip

We shall now see that the method used in Theorem \ref{arc-conjecture} to prove that the normed algebra $(\DD(J),\norm)$ is incomplete for every non-rectifiable Jordan arc $J$ can be developed to apply to some more general sets.

Let $X$ be a compact space, and let $A$ be a uniform algebra on $X$. Recall that a point $z \in X$ is a \emph{peak point} for $A$ if there exists a function $f \in A$  with $f(z)=1$ and such that
$|f(w)|<1$ for all $w \in X \setminus \{z\}$.
Similarly, a non-empty, closed subset $E$ of $X$ is a \emph{peak set} for $A$ if
there exists a function $f \in A$  with $f(E)=\{1\}$ and such that
$|f(w)|<1$ for all $w \in X \setminus E$.

Let $X$ be a compact plane set. Although it is not noted explicitly in \cite{Gamelin}, it is an immediate consequence
\cite[Chapter VIII, Corollary 4.4]{Gamelin} that every point of the outer boundary of $X$ is a peak point for $P(\widehat X)$.
Furthermore, it then follows from \cite[Chapter II, Corollary 12.8]{Gamelin} that every non-empty, finite subset of
the outer boundary of $X$ is a peak set for $P(\widehat X)$. (See also \cite[Lemma 1.6.18]{Stout2}.)

\begin{lemma}
\label{blodges-lemma}
Let $Y$ be a compact plane set, and let $z$ and $w$ be points in the outer boundary of $Y$.
Suppose further that there is  a rectifiable arc $J$ joining $z$ to $w$ in $\widehat Y$.
Take $\eta\geq 0$. Then there is a polynomial $p$ with $p'(z)=p'(w)=0$, with $p(z)=\eta$, with $|p'|_Y<3$, and with $p(w)\in\R$ such that
$p(w)-p(z) > |z-w|/2.$
\end{lemma}
\begin{proof}
Set $X=\widehat Y$. Then $J = \Jim{\gamma}$ for some rectifiable path $\gamma$ joining $z$ to $w$ in $X$.
Since $z$ and $w$ are in the (outer) boundary of $Y$, $\{z,w\}$ is a peak set for $P(X)$.
Choose $g \in P(X)$ with $g(z)=g(w)=1$ and with $|g(u)|<1$ for all $u \in X \setminus \{z,w\}$.
By the dominated convergence theorem,
\[
\lim_{n\to\infty}\int_J|g|^n = 0\,.
\]
Thus, by replacing $g$ by a suitable power $g^n$ if necessary, we may suppose that
\[
\int_J |g| < \frac{|z-w|}{2}\,.
\]
Since $g \in P(X)$, it is easy to check that
there exists a polynomial $q$ with $|q|_X <2$, with $q(z)=q(w)=1$ and with
\[
\int_J |q| < \frac{|z-w|}{2}\,.
\]
Take a polynomial $r$ with $r'=q$, and set $s=r-Z$, so that $s'=q-1$.
Then $|s'|_X < 3$, and
\[
s(z)-s(w) = - \int_J s' = w-z - \int_J q\,,
\]
so that $|s(z)-s(w)| > |z-w|/2$.
Certainly we have $s'(z)=s'(w)=0$.
Choose a real number $\theta$ such that $\e^{\i\theta}(s(w)-s(z))>0$, and
let $p$ be the polynomial
$\eta + \e^{\i\theta}(s-s(z))$. It is clear that the polynomial $p$ has the desired properties.
\end{proof}

Using Lemma \ref{blodges-lemma}, we now prove a result for sets of the type shown in Figure \ref{blodges-figure}, below.
This result may also be used to prove our earlier incompleteness results concerning Jordan arcs.
\begin{figure}[htb]
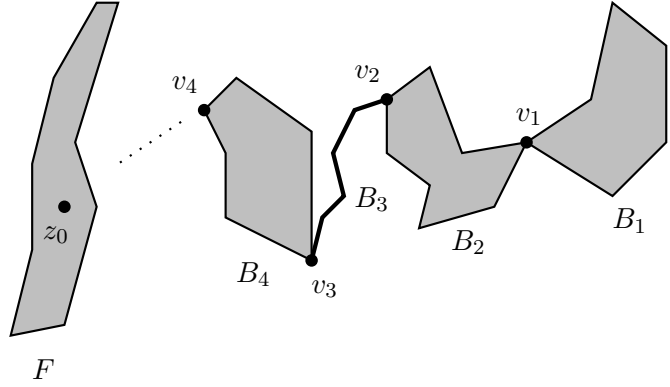
\caption{A set of the type discussed in Theorem \ref{blodges-theorem}}
\label{blodges-figure}
\end{figure}
\begin{theorem}
\label{blodges-theorem}
Let $X$ be a connected, compact plane set, let $F$ be a non-empty, closed subset of  $X$, and let $z_0 \in F$.
Suppose that, for each $n \in \N$, there are a compact plane set $B_n \subseteq X \setminus F$ and a
point $v_n \in B_n$ such that
the following conditions hold:
\begin{enumerate}
\item[(i)] $X = F \cup \bigcup_{n=1}^\infty B_n\,;$\sskip
\item[(ii)]
$B_m \cap B_n = \emptyset$ whenever $m,n \in \N$ with $|m-n|>1$;\sskip
\item[(iii)]
$B_n \cap B_{n+1} = \{v_n\}$ for all $n \in \N$;\sskip
\item[(iv)]
the sequence of sets $(B_n)$ accumulates \jchange only on a subset of $F$, in the sense that
\[
\bigcap_{n=1}^\infty
\overline{\bigcup_{k=n}^\infty B_k} \subseteq F\,;
\]
\item[(v)]
for each $n \in \N$, $v_n$ may be joined to $v_{n+1}$ by some rectifiable arc in $\widehat{B_{n+1}}$;
\sskip
\item[(vi)]
$\sup_{n \in \N}(\sum_{k=n}^\infty|v_{k+1}-v_k|/|z_0-v_n|) = \infty$\,.
\end{enumerate}
Then $\DDXN$ is incomplete.
\end{theorem}
\begin{proof}
We first note that clause (iv) implies that
$F \cup \bigcup_{k=n}^\infty B_k$ is closed for each $n \in \N$.
Working in $\DD(X)$, we show that $Q_{z_0} = \infty$.

Given $C>0$, choose $n,M\in\N$ with $n<M$ and such that
\[
\sum_{k=n}^M \frac{|v_{k+1}-v_k|}{|v_n-z_0|} > 6C\,.
\]
 Set $Y=\bigcup_{k=n+1}^M B_k$.

We now apply Lemma \ref{blodges-lemma} successively to the sets $B_{n+1}, B_{n+2}, \dots, B_M$ (making appropriate choices of $\eta$), and patch the resulting functions together to give a function $f \in \DD(Y)$ with $|f'|_Y<3$ and with
\[
|f(v_{M+1})-f(v_n)| > \frac12 \sum_{k=n}^M |v_{k+1}-v_k|\,.
\]
We then extend $f$ to a function in $C(X)$, also denoted by $f$, such that $f$ is constant on each of the sets $F\cup\bigcup_{k=M+1}^\infty B_k$ and
\jchange $\bigcup_{k=1}^n B_k$.
Then we have $f \in \DD(X)$, $|f'|_X<3$, and $|f(z_0)-f(v_n)| > 3C |z_0-v_n|$.
Thus $Q_{z_0} > C$.

As $C$ is arbitrary, $Q_{z_0}=\infty$, and so $\DDXN$  is incomplete.
\end{proof}

We now wish to discuss a class of sets which our results do not cover
\begin{figure}[htb]
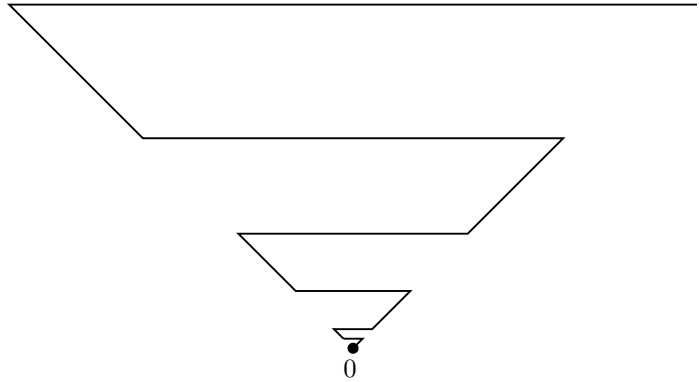
\caption{A Jordan arc repeatedly crossing a triangle}
\label{triangle2}
\end{figure}

Let $Y$ be a triangle in the plane, with one vertex at $0$, and with a horizontal edge above this vertex.
We may then form a Jordan arc $J$ in $Y$ joining the top right corner of $Y$ to $0$ as follows. The arc $J$ consists of an infinite sequence of line segments accumulating at $0$. These line segments alternate between crossing $Y$ horizontally, and then following one of the non-horizontal edges. Finally, we add in the point $0$ at the end.
Such an arc $J$ is shown in Figure \ref{triangle2}, above.

By our earlier results, we know that $(\DD(J),\norm)$ is complete if and only if $J$ is pointwise regular.
Using Theorem \ref{blodges-theorem}, we may show that our conjecture also holds for sets $X$ formed from $J$ by slightly \lq fattening' the horizontal line segments in the manner shown in Figure \ref{triangle3}, below.

\begin{figure}[htb]
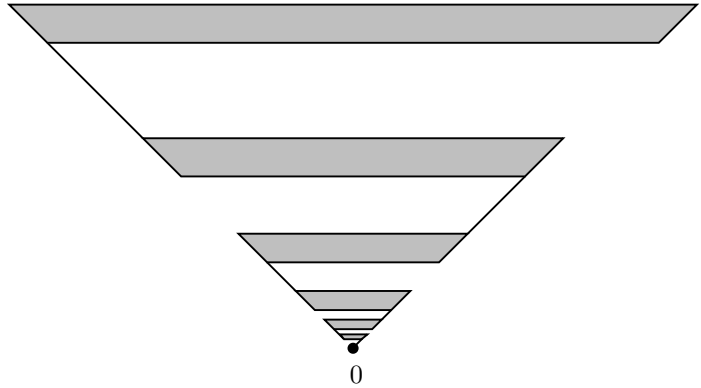
\caption{A set obtained by \lq fattening' horizontal line segments}
\label{triangle3}
\end{figure}

However, none of our results apply to a set $X$ formed by similarly fattening \textit{all} of the line segments forming $J$, to form a set of the type shown in Figure \ref{triangle4}, below.
We call a set of this form a \emph{Superman set}.

\begin{figure}[htb]
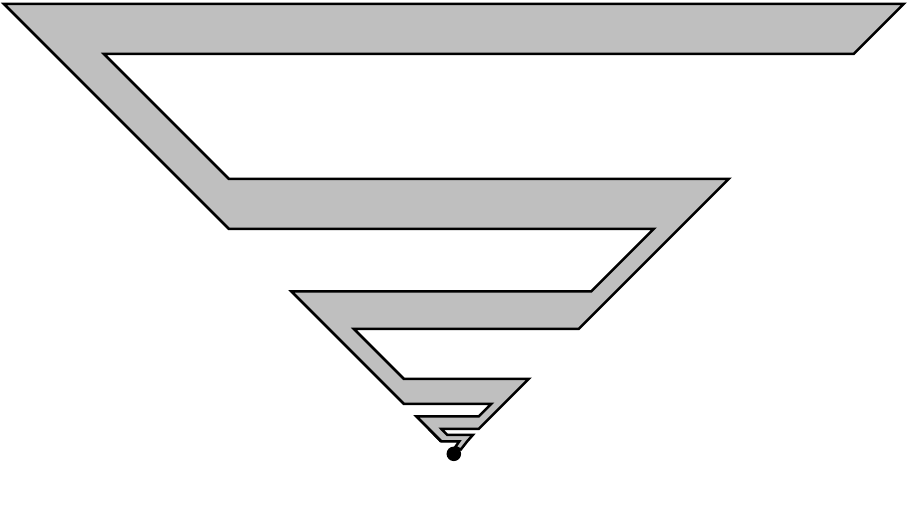
\caption{A \lq fattened' Jordan arc repeatedly crossing a triangle}
\label{triangle4}
\end{figure}

It is possible to show, using Theorem \ref{blodges-theorem} and an inductive argument, \jchange that there are examples of compact plane sets $X$ of this latter type such that $(\DD(X),\norm)$ is incomplete.
However, we do not know whether or not there is a set $X$ of this type such that $X$ is not pointwise regular, and yet $(\DD(X),\norm)$ is complete.
\newpage
\section{Open questions}
\noindent
We conclude with a set of open questions concerning perfect, compact plane sets $X$.
\begin{enumerate}
\item Suppose that $X$ is connected, and that $(\DD(X),\Jnorm)$ is complete. Is it necessarily true that $X$ must be pointwise regular? Is this true if we assume, in addition, that $X$ is a Superman set?
\smallskip
\item Is it always true that $\DD(X)\subseteq R(X)$? Does this hold, at least, whenever $X$ is pointwise regular?
\smallskip
\item Is $R_0(X)$ always dense in $(\DD(X),\Jnorm)$? Is  $P_0(X)$ dense in $\DDXN$ whenever $X$ is polynomially convex?
\smallskip
\item Suppose that $\Tint X$ is dense in $X$. Is it always \jchange true that  $\Aone(X)\subseteq R(X)$?  Is $\Aone(X)$ always natural?
\smallskip
\item Is $\BF{X}$ natural whenever $\paths$ is an effective family of paths in $X$?
\smallskip
\item Suppose that $X$ is semi-rectifiable, and that $\paths$ is the set of all admissible paths in $X$. Is $\wdiff{X}$   always equal to $\BF{X}$?
Suppose further that $\DDXN$  is complete.  Does it follow that
\[
\DD(X)=\BF{X}\,?
\]
\item Is there a uniformly regular, polynomially convex, compact plane set such that
$\Tint X$ is connected and dense in $X$, and yet \[\wdiff{X}=\DD(X)\neq \Aone(X)\,?\]
\item (Local behaviour of the function $Q_X$.) Let $X$ be a connected, compact plane set, and let $z_0 \in X$. Suppose that $z_0$ has a connected, compact neighbourhood $N$ in $X$ such that $Q_N(z_0)=\infty$. Is it necessarily true that $Q_X(z_0)=\infty$?
\end{enumerate}

\end{document}

%% file: counter-mh2.pdf_t
\begin{picture}(0,0)%
\includegraphics{counter-mh2.pdf}%
\end{picture}%
\setlength{\unitlength}{3315sp}%
\begingroup\makeatletter\ifx\SetFigFont\undefined%
\gdef\SetFigFont#1#2#3#4#5{%
  \reset@font\fontsize{#1}{#2pt}%
  \fontfamily{#3}\fontseries{#4}\fontshape{#5}%
  \selectfont}%
\fi\endgroup%
\begin{picture}(4114,4215)(848,-10681)
\put(2701,-10636){\makebox(0,0)[lb]{\smash{{\SetFigFont{12}{14.4}{\familydefault}{\mddefault}{\updefault}{\color[rgb]{0,0,0}$z_2$}%
}}}}
\put(4501,-10636){\makebox(0,0)[lb]{\smash{{\SetFigFont{12}{14.4}{\familydefault}{\mddefault}{\updefault}{\color[rgb]{0,0,0}$z_1=\frac12$}%
}}}}
\put(4501,-6586){\makebox(0,0)[lb]{\smash{{\SetFigFont{12}{14.4}{\familydefault}{\mddefault}{\updefault}{\color[rgb]{0,0,0}$w_1$}%
}}}}
\put(3376,-6586){\makebox(0,0)[lb]{\smash{{\SetFigFont{12}{14.4}{\familydefault}{\mddefault}{\updefault}{\color[rgb]{0,0,0}$w'_1$}%
}}}}
\put(3601,-10636){\makebox(0,0)[lb]{\smash{{\SetFigFont{12}{14.4}{\familydefault}{\mddefault}{\updefault}{\color[rgb]{0,0,0}$z'_1$}%
}}}}
\put(2161,-8431){\makebox(0,0)[lb]{\smash{{\SetFigFont{12}{14.4}{\familydefault}{\mddefault}{\updefault}{\color[rgb]{0,0,0}$w'_2$}%
}}}}
\put(2791,-8431){\makebox(0,0)[lb]{\smash{{\SetFigFont{12}{14.4}{\familydefault}{\mddefault}{\updefault}{\color[rgb]{0,0,0}$w_2$}%
}}}}
\put(2341,-10636){\makebox(0,0)[lb]{\smash{{\SetFigFont{12}{14.4}{\familydefault}{\mddefault}{\updefault}{\color[rgb]{0,0,0}$z'_2$}%
}}}}
\put(901,-10636){\makebox(0,0)[lb]{\smash{{\SetFigFont{12}{14.4}{\familydefault}{\mddefault}{\updefault}{\color[rgb]{0,0,0}$0$}%
}}}}
\put(1801,-10636){\makebox(0,0)[lb]{\smash{{\SetFigFont{12}{14.4}{\familydefault}{\mddefault}{\updefault}{\color[rgb]{0,0,0}$z_3$}%
}}}}
\end{picture}%

%% file: square-BF-example2.pdf_t
\begin{picture}(0,0)%
\includegraphics{square-BF-example2.pdf}%
\end{picture}%
\setlength{\unitlength}{2368sp}%
\begingroup\makeatletter\ifx\SetFigFont\undefined%
\gdef\SetFigFont#1#2#3#4#5{%
  \reset@font\fontsize{#1}{#2pt}%
  \fontfamily{#3}\fontseries{#4}\fontshape{#5}%
  \selectfont}%
\fi\endgroup%
\begin{picture}(6547,6287)(676,-12626)
\put(976,-12586){\makebox(0,0)[lb]{\smash{{\SetFigFont{12}{14.4}{\familydefault}{\mddefault}{\updefault}{\color[rgb]{0,0,0}$0$}%
}}}}
\put(2851,-7261){\makebox(0,0)[lb]{\smash{{\SetFigFont{11}{13.2}{\familydefault}{\mddefault}{\updefault}{\color[rgb]{0,0,0}$r_1$}%
}}}}
\put(676,-7711){\makebox(0,0)[lb]{\smash{{\SetFigFont{11}{13.2}{\familydefault}{\mddefault}{\updefault}{\color[rgb]{0,0,0}$s_{1\,}\i$}%
}}}}
\put(2776,-8461){\makebox(0,0)[lb]{\smash{{\SetFigFont{11}{13.2}{\familydefault}{\mddefault}{\updefault}{\color[rgb]{0,0,0}$r_2$}%
}}}}
\put(676,-8161){\makebox(0,0)[lb]{\smash{{\SetFigFont{11}{13.2}{\familydefault}{\mddefault}{\updefault}{\color[rgb]{0,0,0}$s_{2\,}\i$}%
}}}}
\put(676,-8836){\makebox(0,0)[lb]{\smash{{\SetFigFont{11}{13.2}{\familydefault}{\mddefault}{\updefault}{\color[rgb]{0,0,0}$s_{3\,}\i$}%
}}}}
\put(676,-9361){\makebox(0,0)[lb]{\smash{{\SetFigFont{11}{13.2}{\familydefault}{\mddefault}{\updefault}{\color[rgb]{0,0,0}$s_{4\,}\i$}%
}}}}
\put(676,-9811){\makebox(0,0)[lb]{\smash{{\SetFigFont{11}{13.2}{\familydefault}{\mddefault}{\updefault}{\color[rgb]{0,0,0}$s_{5\,}\i$}%
}}}}
\put(676,-10411){\makebox(0,0)[lb]{\smash{{\SetFigFont{11}{13.2}{\familydefault}{\mddefault}{\updefault}{\color[rgb]{0,0,0}$s_{6\,}\i$}%
}}}}
\put(2631,-10736){\makebox(0,0)[lb]{\smash{{\SetFigFont{11}{13.2}{\familydefault}{\mddefault}{\updefault}{\color[rgb]{0,0,0}$r_3$}%
}}}}
\put(2196,-10881){\makebox(0,0)[lb]{\smash{{\SetFigFont{11}{13.2}{\familydefault}{\mddefault}{\updefault}{\color[rgb]{0,0,0}$r_4$}%
}}}}
\end{picture}%

%% file: RSA-example.pdf_t
\begin{picture}(0,0)%
\includegraphics{RSA-example.pdf}%
\end{picture}%
\setlength{\unitlength}{2171sp}%
\begingroup\makeatletter\ifx\SetFigFont\undefined%
\gdef\SetFigFont#1#2#3#4#5{%
  \reset@font\fontsize{#1}{#2pt}%
  \fontfamily{#3}\fontseries{#4}\fontshape{#5}%
  \selectfont}%
\fi\endgroup%
\begin{picture}(8242,7231)(1200,-10596)
\put(6781,-6810){\makebox(0,0)[lb]{\smash{{\SetFigFont{11}{13.2}{\sfdefault}{\mddefault}{\updefault}{\color[rgb]{0,0,0}$z_3$}%
}}}}
\put(7279,-5455){\makebox(0,0)[lb]{\smash{{\SetFigFont{11}{13.2}{\familydefault}{\mddefault}{\updefault}{\color[rgb]{0,0,0}$a_1$}%
}}}}
\put(7775,-6126){\makebox(0,0)[lb]{\smash{{\SetFigFont{11}{13.2}{\sfdefault}{\mddefault}{\updefault}{\color[rgb]{0,0,0}$a_2$}%
}}}}
\put(8383,-5657){\makebox(0,0)[lb]{\smash{{\SetFigFont{11}{13.2}{\familydefault}{\mddefault}{\updefault}{\color[rgb]{0,0,0}$w_2$}%
}}}}
\put(7544,-4477){\makebox(0,0)[lb]{\smash{{\SetFigFont{11}{13.2}{\familydefault}{\mddefault}{\updefault}{\color[rgb]{0,0,0}$w_1$}%
}}}}
\put(8534,-7010){\makebox(0,0)[lb]{\smash{{\SetFigFont{11}{13.2}{\familydefault}{\mddefault}{\updefault}{\color[rgb]{0,0,0}$z_0$}%
}}}}
\put(6046,-6242){\makebox(0,0)[lb]{\smash{{\SetFigFont{11}{13.2}{\familydefault}{\mddefault}{\updefault}{\color[rgb]{0,0,0}$z_1$}%
}}}}
\put(6355,-6633){\makebox(0,0)[lb]{\smash{{\SetFigFont{11}{13.2}{\sfdefault}{\mddefault}{\updefault}{\color[rgb]{0,0,0}$z_2$}%
}}}}
\put(4912,-7205){\makebox(0,0)[lb]{\smash{{\SetFigFont{11}{13.2}{\familydefault}{\mddefault}{\updefault}{\color[rgb]{0,0,0}$0$}%
}}}}
\put(8981,-6417){\makebox(0,0)[lb]{\smash{{\SetFigFont{11}{13.2}{\sfdefault}{\mddefault}{\updefault}{\color[rgb]{0,0,0}$a_3$}%
}}}}
\put(8503,-6686){\makebox(0,0)[lb]{\smash{{\SetFigFont{11}{13.2}{\sfdefault}{\mddefault}{\updefault}{\color[rgb]{0,0,0}$w_4$}%
}}}}
\put(8491,-6187){\makebox(0,0)[lb]{\smash{{\SetFigFont{11}{13.2}{\sfdefault}{\mddefault}{\updefault}{\color[rgb]{0,0,0}$w_3$}%
}}}}
\end{picture}%

%% file: square-crossed.pdf_t
\begin{picture}(0,0)%
\includegraphics{square-crossed.pdf}%
\end{picture}%
\setlength{\unitlength}{2763sp}%
\begingroup\makeatletter\ifx\SetFigFont\undefined%
\gdef\SetFigFont#1#2#3#4#5{%
  \reset@font\fontsize{#1}{#2pt}%
  \fontfamily{#3}\fontseries{#4}\fontshape{#5}%
  \selectfont}%
\fi\endgroup%
\begin{picture}(4073,3813)(1951,-7752)
\put(2164,-7712){\makebox(0,0)[lb]{\smash{{\SetFigFont{11}{13.2}{\familydefault}{\mddefault}{\updefault}{\color[rgb]{0,0,0}$0$}%
}}}}
\put(1951,-5761){\makebox(0,0)[lb]{\smash{{\SetFigFont{11}{13.2}{\familydefault}{\mddefault}{\updefault}{\color[rgb]{0,0,0}$y_{1\,}\i$}%
}}}}
\put(1957,-6301){\makebox(0,0)[lb]{\smash{{\SetFigFont{11}{13.2}{\familydefault}{\mddefault}{\updefault}{\color[rgb]{0,0,0}$y_{2\,}\i$}%
}}}}
\put(1957,-6781){\makebox(0,0)[lb]{\smash{{\SetFigFont{11}{13.2}{\familydefault}{\mddefault}{\updefault}{\color[rgb]{0,0,0}$y_{3\,}\i$}%
}}}}
\put(1951,-7036){\makebox(0,0)[lb]{\smash{{\SetFigFont{11}{13.2}{\familydefault}{\mddefault}{\updefault}{\color[rgb]{0,0,0}$y_{4\,}\i$}%
}}}}
\end{picture}%

%% file: blodges.pdf_t
\begin{picture}(0,0)%
\includegraphics{blodges.pdf}%
\end{picture}%
\setlength{\unitlength}{3552sp}%
\begingroup\makeatletter\ifx\SetFigFont\undefined%
\gdef\SetFigFont#1#2#3#4#5{%
  \reset@font\fontsize{#1}{#2pt}%
  \fontfamily{#3}\fontseries{#4}\fontshape{#5}%
  \selectfont}%
\fi\endgroup%
\begin{picture}(4683,2687)(-2271,-12101)
\put(826,-11161){\makebox(0,0)[lb]{\smash{{\SetFigFont{11}{13.2}{\familydefault}{\mddefault}{\updefault}{\color[rgb]{0,0,0}$B_2$}%
}}}}
\put(1951,-11011){\makebox(0,0)[lb]{\smash{{\SetFigFont{11}{13.2}{\familydefault}{\mddefault}{\updefault}{\color[rgb]{0,0,0}$B_1$}%
}}}}
\put(-149,-11461){\makebox(0,0)[lb]{\smash{{\SetFigFont{11}{13.2}{\familydefault}{\mddefault}{\updefault}{\color[rgb]{0,0,0}$v_3$}%
}}}}
\put(1276,-10261){\makebox(0,0)[lb]{\smash{{\SetFigFont{11}{13.2}{\familydefault}{\mddefault}{\updefault}{\color[rgb]{0,0,0}$v_1$}%
}}}}
\put(151,-9961){\makebox(0,0)[lb]{\smash{{\SetFigFont{11}{13.2}{\familydefault}{\mddefault}{\updefault}{\color[rgb]{0,0,0}$v_2$}%
}}}}
\put(-1124,-10036){\makebox(0,0)[lb]{\smash{{\SetFigFont{11}{13.2}{\familydefault}{\mddefault}{\updefault}{\color[rgb]{0,0,0}$v_4$}%
}}}}
\put(-674,-11386){\makebox(0,0)[lb]{\smash{{\SetFigFont{11}{13.2}{\familydefault}{\mddefault}{\updefault}{\color[rgb]{0,0,0}$B_4$}%
}}}}
\put(151,-10861){\makebox(0,0)[lb]{\smash{{\SetFigFont{11}{13.2}{\familydefault}{\mddefault}{\updefault}{\color[rgb]{0,0,0}$B_3$}%
}}}}
\put(-2024,-11086){\makebox(0,0)[lb]{\smash{{\SetFigFont{11}{13.2}{\familydefault}{\mddefault}{\updefault}{\color[rgb]{0,0,0}$z_0$}%
}}}}
\put(-2099,-12061){\makebox(0,0)[lb]{\smash{{\SetFigFont{11}{13.2}{\familydefault}{\mddefault}{\updefault}{\color[rgb]{0,0,0}$F$}%
}}}}
\end{picture}%

%% file: triangle2.pdf_t
\begin{picture}(0,0)%
\includegraphics{triangle2.pdf}%
\end{picture}%
\setlength{\unitlength}{3158sp}%
\begingroup\makeatletter\ifx\SetFigFont\undefined%
\gdef\SetFigFont#1#2#3#4#5{%
  \reset@font\fontsize{#1}{#2pt}%
  \fontfamily{#3}\fontseries{#4}\fontshape{#5}%
  \selectfont}%
\fi\endgroup%
\begin{picture}(5444,2987)(2079,-9026)
\put(4726,-8986){\makebox(0,0)[lb]{\smash{{\SetFigFont{10}{12.0}{\familydefault}{\mddefault}{\updefault}{\color[rgb]{0,0,0}$0$}%
}}}}
\end{picture}%

%% file: triangle3.pdf_t
\begin{picture}(0,0)%
\includegraphics{triangle3.pdf}%
\end{picture}%
\setlength{\unitlength}{3158sp}%
\begingroup\makeatletter\ifx\SetFigFont\undefined%
\gdef\SetFigFont#1#2#3#4#5{%
  \reset@font\fontsize{#1}{#2pt}%
  \fontfamily{#3}\fontseries{#4}\fontshape{#5}%
  \selectfont}%
\fi\endgroup%
\begin{picture}(5444,3036)(2679,-11475)
\put(5376,-11435){\makebox(0,0)[lb]{\smash{{\SetFigFont{10}{12.0}{\familydefault}{\mddefault}{\updefault}{\color[rgb]{0,0,0}$0$}%
}}}}
\end{picture}%

%% file: triangle4.pdf_t
\begin{picture}(0,0)%
\includegraphics{triangle4.pdf}%
\end{picture}%
\setlength{\unitlength}{3158sp}%
\begingroup\makeatletter\ifx\SetFigFont\undefined%
\gdef\SetFigFont#1#2#3#4#5{%
  \reset@font\fontsize{#1}{#2pt}%
  \fontfamily{#3}\fontseries{#4}\fontshape{#5}%
  \selectfont}%
\fi\endgroup%
\begin{picture}(5444,3062)(2079,-9101)
\put(4726,-9061){\makebox(0,0)[lb]{\smash{{\SetFigFont{10}{12.0}{\familydefault}{\mddefault}{\updefault}{\color[rgb]{0,0,0}$0$}%
}}}}
\end{picture}%